\documentclass[11pt,reqno]{amsart}
\newcommand{\mysection}[1]{\section{#1}
      \setcounter{equation}{0}}

\newcommand\cbrk{\text{$]$\kern-.15em$]$}} 
\newcommand\opar{\text{\raise.2ex\hbox{${\scriptstyle | }$}\kern-.34em$($} }

\usepackage{color}

\DeclareMathOperator*{\esssup}{ess\,sup}

\newtheorem{theorem}{Theorem}[section]
\newtheorem{lemma}[theorem]{Lemma}

\theoremstyle{definition}
\newtheorem{assumption}{Assumption}[section]
\newtheorem{definition}{Definition}[section]

\theoremstyle{remark}
\newtheorem{remark}{Remark}[section]

\newcommand\bH{\mathbb{H}}

\newcommand\bR{\mathbb{R}}

\newcommand\bT{\mathbb{T}}

\newcommand\cB{\mathcal{B}}

\newcommand\cF{\mathcal{F}}

\newcommand\cH{\mathcal{H}}

\newcommand\cP{\mathcal{P}}

\makeatletter
 \newcommand{\sumstar}%
 {\operatornamewithlimits{\sum@\kern-.2em\raise1ex\hbox{*}}}
 \makeatother

\begin{document}

\title[Degenerate stochastic PDEs]{
On the solvability of degenerate stochastic partial differential equations in Sobolev spaces}

\author[M. Gerencs\'er]{M\'at\'e Gerencs\'er}
\address{School of Mathematics and Maxwell Institute,
University of Edinburgh,
King's  Buildings,
Edinburgh, EH9 3JZ, United Kingdom}
\email{m.gerencser@sms.ed.ac.uk}

\author[I. Gy\"ongy]{Istv\'an Gy\"ongy}
\address{School of Mathematics and Maxwell Institute,
University of Edinburgh,
King's  Buildings,
Edinburgh, EH9 3JZ, United Kingdom}
\email{gyongy@maths.ed.ac.uk}

\author[N.V. Krylov]{Nicolai Krylov}%
\thanks{The work of the 
third author was partially supported
by NSF grant DMS-1160569}
\address{127 Vincent Hall, University of Minnesota,
Minneapolis,
       MN, 55455, USA}
\email{krylov@math.umn.edu}


\begin{abstract} Systems of parabolic, possibly degenerate parabolic 
SPDEs are considered. Existence and uniqueness are established in Sobolev spaces. Similar results are obtained for a class of equations generalizing the deterministic first order symmetric 
hyperbolic systems. 
\end{abstract}

\maketitle
 \mysection{introduction}
 
 In this paper we are interested in the solvability in $L_p$ spaces 
of linear stochastic parabolic, possibly degenerate, 
 PDEs and of 
systems of linear stochastic parabolic PDEs. The equations we consider are 
important in applications. They arise, in nonlinear filtering 
of partially observable stochastic processes, in modelling of hydromagnetic dynamo 
evolving in fluids with random velocities, 
and in many other areas of physics and engineering. 

An $L_2$-theory of degenerate linear elliptic and parabolic PDEs is developed in 
\cite{Ol1}, \cite{Ol2}, \cite{OR1} and \cite{OR2}. The solvability in $L_2$ spaces of linear degenerate stochastic PDEs 
of parabolic type were first studied in \cite{KR1977} 
(see also \cite{Ro}). The first existence and uniqueness theorem
on solvability of these equations in $W^m_p$ spaces is presented in \cite{KR}. This result 
is improved in \cite{GK2003}. 

In the present paper we fill in a gap in the proof of the existence and uniqueness theorems 
in \cite{KR} and \cite{GK2003}. Moreover, we essentially improve these theorems  
(see Theorem \ref{theorem main}),  and 
our main result, Theorem \ref{theorem systemmain}, extends them to degenerate stochastic 
parabolic systems.  We present also an existence and uniqueness theorem, 
Theorem \ref{theorem antisymmetric},  
on solvability in $W^m_2$ spaces for a larger class of stochastic parabolic systems, 
which, in particular, contains the first order symmetric hyperbolic systems. 
This result was indicated in \cite{GK2006}. 

Finally we would like to mention that for some special degenerate stochastic PDEs, 
for example for the stochastic Euler equations, there are many results on solvability 
in the literature. (See, for example, \cite{BFM} and the references therein.)

In conclusion we introduce some notation used throughout the paper. 
All random elements will be given on a fixed
  probability space 
$(\Omega,\cF,P)$, equipped with a filtration 
$(\cF_t)_{t\geq0}$   of  $\sigma$-fields
$\cF_{t}\subset\cF$. We suppose that
this probability space carries  
a sequence of independent Wiener processes $(w^r)_{r=1}^{\infty}$,  
adapted to the filtration $(\cF_t)_{t\geq0}$, such that 
$w^r_t-w^r_s$ is independent of $\cF_s$ for each $r$ and any $0\leq s\leq t$.
It is assumed that $\cF_0$ contains all $P$-null subsets of $\Omega$, 
 so that
$(\Omega,\cF,P)$ is a complete probability space and the $\sigma$-fields
$\cF_{t}$ are complete. 
By $\mathcal P$ we denote the predictable
$\sigma$-field of subsets of $\Omega\times(0,\infty)$
generated by  $(\cF_t)_{t\geq0}$.

For $p\in[1,\infty)$, the space of measurable mappings 
$f$ from $\bR^d$ into a separable Hilbert space $\cH$, such that 
$$
\|f\|_{L_{p}}=
\big(\int_{\bR^d}|f(x)|_{\cH}^p\,dx\big)^{1/p}<\infty,
$$
 is denoted by 
$L_p(\bR^{d},\cH)$.

\begin{remark}
                                       \label{remark 4.11.1}
We did not include
the symbol $\cH$ in the notation of the   norm
in $L_{p}(\bR^{d},\cH)$. Which $\cH$ is involved will be absolutely clear 
from  the context. We  do the same in
other similar situations.
\end{remark} 

Often $\cH$ will be $l_2$, 
or the space of infinite matrices  $\{g^{ij}\in\bR:i=1,...,M,\, j=1,2,...\}$,  
or finite 
$M\times M$ matrices with the Hilbert-Schmidt norm.  
The space of functions from $L_p(\bR^{d},\cH)$, 
whose generalized derivatives 
up to order $m$ are also in $L_p(\bR^{d},\cH)$, is denoted by $W^m_p
(\bR^{d},\cH)$. 
By definition $W^0_p(\bR^{d},\cH)=L_p(\bR^{d},\cH)$. 
The norm $|u|_{W^m_p}$ of $u$ in $W^m_p(\bR^{d},\cH)$ 
is defined 
by  
\begin{equation}
                                          \label{4.11.1}
|u|^p_{W^m_p}=\sum_{|\alpha|\leq m}|D^{\alpha}u|^p_{L_p}, 
\end{equation}
where 
$D^{\alpha}:=D_1^{\alpha_1}...D_d^{\alpha_d}$ for 
{\it multi-indices} 
$\alpha:=(\alpha_1,...,\alpha_d)\in\{0,1,...\}^d$ of 
{\it length} $|\alpha|:=\alpha_1+\alpha_2+...+\alpha_d$,  
and $D_iu$ is the generalized derivative of $u$ with respect to 
$x^i$ 
for $i=1,2...,d$. We also use the notation $D_{ij}=D_iD_j$ and   
$Du=(D_1u,...,D_du)$. 
 Unless otherwise indicated, 
the summation convention with respect to repeated integer 
valued indices is used throughout the paper.  

\mysection{formulation}

                                            \label{section 1}
In this section $\cH=\bR$ and we use a shorter notation
$$
L_{p}=L_{p}(\bR^{d},\bR),\quad W^{m}_{p}=W^{m}_{p}(\bR^{d},\bR),
\quad W^{m+1}_p
 (l_2) =W^{m+1}_p
 (\bR^{d},l_2). 
$$

 Fix a $T\in(0,\infty)$ and  consider the problem 
\begin{equation}                                      \label{equation}
du_t(x)=(L_tu_t(x)+f_t(x))\,dt+(M^r_tu_t(x)+g^r_t(x))\,dw^r_t, 
\end{equation}
$(t,x)\in H_T:=[0,T]\times\mathbb R^d$,  with initial condition
\begin{equation}                                                                 \label{ini}
 u_0(x)=\psi(x),\quad x\in\mathbb R^d, 
\end{equation}
where 
$$
L_t=a^{ij}_t(x)D_{ij}+b^{i}_t(x)D_i+c_t(x), 
\quad M^r_t=\sigma_t^{ir}(x)D_i+\nu^r_t(x),
$$
and all functions, given
on $\Omega\times H_T$, are assumed to be real valued
and satisfy the following assumptions in which
$m\geq0$
 is an integer and $K$ is a constant.  

\begin{assumption}                     \label{assumption regularitycoeff}
The derivatives in $x\in\bR^d$ of $a^{ij}$ up to order $\max(m,2)$ and of $b^{i}$ and $c$ 
up to order $m$ are $\mathcal P\otimes\mathcal B(\bR^d)$-measurable 
functions, bounded by $K$ for all $i,j\in\{1,2,...,d\}$. 
The functions   $\sigma^{i}
=(\sigma^{ir})_{r=1}^{\infty}$ and $\nu=(\nu^{r})_{r=1}^{\infty}$ 
are $l_2$-valued and their derivatives in $x$ 
up to order $m+1$ are 
$\mathcal P\otimes\mathcal B(\bR^d)$-measurable 
$l_2$-valued functions, bounded by $K$.  
\end{assumption}

\begin{assumption}                       \label{assumption regularitydata}
The free data, $ f_t $ and $ g_t =(g^r)_{r=1}^{\infty}$ are 
predictable 
processes with values in $W^m_p$ and $W^{m+1}_p
 (l_2) $, respectively, 
such that almost surely
\begin{equation}															
\mathcal K_{m,p}^p(T)=\int_0^T
\big(|f_t|^p_{W^m_p}+|g_t|^p_{W^{m+1}_p }\big)\,dt<\infty.
\end{equation}
The initial value, $\psi$ is an $\cF_0$-measurable random variable with values in $W^m_p$. 
\end{assumption}

\begin{assumption}                                                         \label{assumption parabolicity}
For $P\otimes dt\otimes dx$-almost all $(\omega,t,x)\in\Omega\times[0,T]\times\bR^d$ 
$$
 \alpha^{ij}_t(x)z^iz^j\geq 0
$$
for all $z\in\bR^d$, where 
$$
\alpha^{ij}=2a^{ij}-\sigma^{ir}\sigma^{ jr}. 
$$
\end{assumption}

Let $\tau$ be a stopping time bounded by $T$. 
\begin{definition}                                                               \label{definition solution}
A $W_p^1$-valued function $u$,
defined on the stochastic interval $\opar0,\tau\cbrk$, is called a solution 
of 
\eqref{equation}-\eqref{ini} 
on $[0,\tau]$ if $u$ is predictable on $\opar0,\tau\cbrk$, 
$$
\int_0^{\tau}|u_t|_{W^1_p}^p<\infty\,(a.s.), 
$$
and for each $\varphi\in C_0^{\infty}(\bR^d)$ for almost all $\omega\in\Omega$
\begin{align*}
(u_t,\varphi)=&(\psi,\varphi)
+\int_0^t\{-(a^{ij}_sD_iu_s,D_j\varphi)
+(\bar b^{i}_sD_iu_s+c_su_s+f_s,\varphi)\}\,ds\\
&+\int_0^t(\sigma^{ir}_sD_iu_s+\nu^{r}_{s}u_s+g^r_s,\varphi)\,dw^r_{s}
\end{align*}
for all $t\in[0,\tau(\omega)]$,   where $\bar
b^i=b^{i}-D_ja^{ij}$, and $(\cdot,\cdot)$ denotes the inner
product in the Hilbert space of square integrable real-valued
functions on $\bR^d$.  
\end{definition}

We want to prove the following result. 
\begin{theorem}                                                                       \label{theorem main}
Let Assumptions 
\ref{assumption regularitycoeff}, 
\ref{assumption regularitydata},
and \ref{assumption parabolicity} hold and $m\geq1$. 
Then there exists a unique solution $u=(u_t)_{t\in[0,T]}$ on $[0,T]$.  
Moreover, $u$ is a  $W^{m}_p$-valued weakly continuous process,  
it is a strongly continuous 
process with values in $W^{m-1}_p$, and for every $q>0$ and $n\in\{0,1,...,m\}$
\begin{equation}                                                                      \label{estimate}
E\sup_{t\in[0,T]}|u_t|_{W^n_p}^q
\leq N(E|\psi|^q_{W^n_p}+E\mathcal K^q_n(T)), 
\end{equation}
where $N$ is a constant depending only on $K$, $T$, $d$, $m$, $p$ and $q$.  
\end{theorem}
This result is proved in \cite{KR} in the case $q=p\geq2$ 
under the additional assumptions 
that $E\mathcal K_{m,p'}(T)<\infty$ and $E|\psi|^{p'}_{W^{m}_{p'}}<\infty$ for 
$p'=p$ and $p'=2$ (see Theorem 3.1 there). These additional assumptions are not 
supposed in \cite{GK2003} and a somewhat weaker version of the above  theorem 
is obtained in \cite{GK2003} when $q\in(0,p]$.  
The proof of it in \cite{GK2003} uses  Theorem 3.1 from \cite{KR}, 
whose proof is based on an estimate for the derivatives of the solution $u$, 
formulated as Lemma 2.1 in \cite{KR}.   
The proof of this lemma, however, contains 
a gap. Our aim is to fill in this gap and also to improve the existence and uniqueness theorems from 
\cite{KR} and \cite{GK2003}. Since $Du=(D_1u,...,D_du)$ satisfies a system of SPDEs, 
it is natural to present and prove our results in the context of systems of stochastic PDEs. 

\mysection{Systems of stochastic PDEs}

Let $M\geq1$ be an integer, and let $\langle\cdot,
\cdot\rangle$ and $\langle\cdot 
 \rangle$ denote the scalar product and the norm
in $\bR^{M}$, respectively.
By $\bT^{M}$ we denote the set
of $M\times M$ matrices, which we consider
as a Euclidean space $\bR^{M^{2}}$. 
For an integer $m\geq 1$ we define $l_{2}(\bR^{m})$
as the space of sequences
  $\nu=(\nu^{1},\nu^{2},...)$
with $\nu^{k}\in\bR^{m}$, $k\geq1$, and finite norm
$$
\|\nu\|_{l_{2} }=\big(
\sum_{k=1}^{\infty}|\nu|^{2}_{\bR^{m}}\big)^{1/2}
$$
(cf. Remark \ref{remark 4.11.1}).

We look for $\bR^{M}$-valued functions $u_t(x)=(u^1_t(x),...,u^M_t(x))$, 
of $\omega\in\Omega$, $t\in[0,T]$ and $x\in\bR^d$, which 
satisfy  the system of equations 
\begin{align}                        
du_{t}=&
[ a^{ij}_{t}D_{ij}u_{t}+b^{i}_{t}D_{i} u_{t} 
 +cu_t+f_{t}]\,dt\nonumber\\
 &+
 [\sigma^{ik}_{t}D_{i}u_{t}+\nu^{k}_{t}u_{t}
+g^{k}_{t}]\,dw^{k}_{t},             \label{system}
\end{align}
and the initial condition 
\begin{equation}
u_0=\psi,                                    \label{systemini}
\end{equation}
where
$
a_{t}=(a^{ij}_{t}(x))
$ 
takes values in the set of
 $d\times d$ symmetric matrices, 
$$ 
\sigma^{i}_{t}=(\sigma^{ik}_{t}(x),k\geq1)\in l_{2} ,
\quad b^{i}_{t}(x)\in\bT^{M}, 
\quad
c_t(x)\in\bT^{M}, 
$$ 
\begin{equation}
                                                      \label{4.12.1}
\nu_{t}(x)\in l_{2}(\bT^{M}),
\quad f_t(x)\in\bR^M, \quad   
g_t(x)\in l_{2}(\bR^{M})
\end{equation}
for $i=1,...,d$, for all $\omega\in\Omega$, $t\geq0$, $x\in\bR^d$.

Let $m$ be a nonnegative integer, $p\in[2,\infty)$ and make the following assumptions, which are 
straightforward adaptations of Assumptions 
\ref{assumption regularitycoeff} 
and \ref{assumption regularitydata}.  

\begin{assumption}                           \label{assumption systemcoeff}
The derivatives in $x\in\bR^d$ of $a^{ij}$ up to order $\max(m,2)$ 
(including the zeroth-order derivative) and 
of $b^{i}$ and $c$ 
up to order $m$ are $\mathcal P\otimes\mathcal B(\bR^d)$-measurable 
functions, in magnitude bounded by $K$ for all $i,j\in\{1,2,...,d\}$. The derivatives in $x$ of 
the $l_2 $-valued functions $\sigma^{i}=(\sigma^{ik})_{k=1}^{\infty}$ 
and the $l_{2}(\bT^{M})$-valued 
function $\nu$ up to order $m+1$ are 
$\mathcal P\otimes\mathcal B(\bR^d)$-measurable
 $l_2$-valued and 
$l_{2}(\bT^{M})$-valued functions, respectively, in magnitude bounded by $K$.  
\end{assumption}
\begin{assumption}                      \label{assumption systemdata} 
The free data, $(f_t)_{t\in[0,T]}$ and $(g_t)_{t\in[0,T]}$ are 
predictable 
processes with values in 
$$
W^m_p(\bR^{d},\bR^M)\quad\text{and}\quad 
 W^{m+1}_p(\bR^{d},l_{2}(\bR^{M}) ),
$$
 respectively, 
such that almost surely
\begin{equation}													\label{02.25}
\mathcal K_{m,p}^p(T)=\int_0^T
\big(|f_t|^p_{W^m_p}+|g_t|^p_{W^{m+1}_p}\big)\,dt<\infty.
\end{equation}
The initial value, $\psi$ is an $\cF_0$-measurable random variable with 
values in $ 
W^m_p(\bR^{d},\bR^M)$. 
\end{assumption}
Set 
$$
\beta^i=b^i-\sigma^{ir}\nu^r, \quad i=1,2,...,d. 
$$ 
Instead of Assumption \ref{assumption parabolicity} we impose now 
the following condition.

\begin{assumption}                 \label{assumption systemparabolicity}
There exist a constant 
$K_0>0$ and a $\cP\times\mathcal B(\bR^d)$-measurable 
$\bR^M$-valued bounded function $h=(h^{i}_t(x))$, whose first order 
derivatives in $x$ are bounded functions, such that 
for all  $\omega\in\Omega$, $t\geq0$ and $x\in\bR^d$ 
\begin{equation}                                                                          \label{h}
|h|+|Dh|\leq K, 
\end{equation}
and for all $(\lambda_1,...,\lambda_d)\in\bR^d$  
\begin{equation}                                                                           
                                               \label{semidefinite 1}                        
|\sum_{i=1}^d(\beta^{ikl}-\delta^{kl}h^i)\lambda_i|^2
\leq K_0\sum_{i,j=1}^d\alpha^{ij}\lambda_i\lambda_j\quad\text{for
$k,l=1,...,M$.}
\end{equation}
\end{assumption}

\begin{remark}                        \label{remark reformulation} 
Notice that condition  \eqref{semidefinite 1} in 
Assumption \ref{assumption systemparabolicity} 
can be reformulated as follows: 
There exists a constant $K_0$ such that 
for all values of the arguments and
all continuously differentiable $\bR^{M}$-valued
functions $u=u(x)$ on $\bR^{d}$ we have
\begin{equation}
                                            \label{12.19.1}
 \langle  u, b^{i}D_{i}u\rangle
-\sigma^{ik}\langle u,\nu^{k}D_{i}u\rangle
\leq K_{0}\big|\sum_{i,j=1}^{d}\alpha^{ij}\langle
D_{i}u,D_{j}u\rangle
\big|^{1/2}\langle u\rangle +h^{i} 
\langle D_{i} u,u \rangle.
\end{equation}

Indeed, set $\hat \beta^{i}=\beta^{i}-h^{i}I_{M}$,
where $I_{M}$ is the $M\times M$ unit matrix and observe that,
  \eqref{12.19.1} means that
$$
\langle u, \hat \beta^{i}D_{i}u\rangle
\leq K_{0}\big|\sum_{i,j=1}^{d}\alpha^{ij}\langle
D_{i}u,D_{j}u\rangle
\big|^{1/2}\langle u\rangle.
$$
By considerng this relation at a fixed point $x$ and noting that
then one can choose $u$ and $Du$ independently, we conclude that
\begin{equation}
                                            \label{4.12.3}
\langle  \sum_{i}\hat \beta^{i}D_{i}u\rangle^{2}\leq K_{0}^{2}
 \alpha^{ij}\langle
D_{i}u,D_{j}u\rangle
\end{equation}
and \eqref{semidefinite 1} follows (with a different $K_{0}$)
if we take $D_{i}u^{k}=\lambda_{i}\delta^{kl}$.

On the other hand, \eqref{semidefinite 1}
means that for any $l$ without summation on $l$
$$
\big|\sum_{i}\hat \beta^{ikl}D_{i}u ^{l} \big|^{2}
\leq K_{0}\alpha^{ij} 
(D_{i}u ^{l} ) D_{j}u ^{l} .
$$
  But then by Cauchy's
inequality similar estimate holds after summation on $l$
is done and carried inside the square on the left-hand side.
This yields \eqref{4.12.3} (with a different constant
$K_{0}$) and then leads to \eqref{12.19.1}.

\end{remark}

\begin{remark}                                                   \label{remark 1.20.3.14}
Notice that, given Assumption \ref{assumption systemcoeff} we can just take 
 $h^i=\beta^i$ for $i=1,...,d$  if
$M=1$ to see that in that case
 Assumption \ref{assumption systemparabolicity} 
 is equivalent to $\alpha\geq0$. 
If Assumption \ref{assumption systemcoeff} holds, 
then  for every $k,l=1,2,...,M$ 
$$
|\sum_{i=1}^d(\beta^{ikl}-\delta^{kl}h^i)\lambda_i|^2
\leq N\sum_{i=1}^d|\lambda_i|^2
$$
with a constant $N=N(K,d)$. Hence if in addition 
to Assumption \ref{assumption systemcoeff} we have 
$\alpha\geq \kappa I_d$ with a constant $\kappa>0$ for all $\omega$, $t\geq0$ 
and $x\in\bR^d$, 
then Assumption \ref{assumption systemparabolicity} holds for any $M\geq1$. 
\end{remark}

The notion of solution to \eqref{system}-\eqref{systemini} is 
a straightforward adaptation of Definition \ref{definition solution} 
to systems of equations. Namely, $u=(u^1,...,u^M)$ is a solution 
on $[0,\tau]$, for a stopping time $\tau\leq T$,  if it is a $W^1_p
(\bR^{d},\bR^{M})$-valued predictable function 
on  $\opar0,\tau\cbrk$, 
$$
\int_{0}^{\tau}|u_t|_{W^1_p}^p\,dt<\infty  \quad \text{(a.s.)}, 
$$
and for each $\bR^M$-valued $\varphi=(\varphi^1,...,\varphi^M)$ 
from $C_0(\bR^d)$ with probability one 
\begin{align}
(u_t,\varphi)=&(\psi,\varphi)
+\int_0^t\{-(a^{ij}_sD_iu_s,D_j\varphi)                                                \nonumber\\
&+(\bar b^{i}_sD_iu_s+c_su_s+f_s,\varphi)\}\,ds\\
&+\int_0^t(\sigma^{ir}_sD_iu_s+\nu^{r}_su_s+g^{r}(s),\varphi)\,dw^r_{s}     \label{systemsolution}
\end{align}
for all $t\in[0, \tau ]$, where 
$\bar b^{i}=b^{i}-D_ja^{ij}I_{M}$. Here, and later on $(\Psi,\Phi)$ denotes the 
inner product in the $L_2$-space of $\bR^M$-valued functions $\Psi$ and $\Phi$ 
defined on $\bR^d$. 

The main result of the paper reads now just like Theorem 
\ref{theorem main} above. 

\begin{theorem}                          \label{theorem systemmain}
 Let Assumptions 
 \ref{assumption systemcoeff}, 
 \ref{assumption systemdata} and 
 \ref{assumption systemparabolicity} hold with $m\geq1$. 
Then there is a unique solution $u=(u^l)_{l=1}^M$ 
to \eqref{system}-\eqref{systemini} on $[0,T]$.  Moreover, $u$ is a weakly continuous 
$W^m_p(\bR^{d},\bR^{M})$-valued process, it is strongly continuous as a $W^{m-1}_p(\bR^{d},\bR^{M})$-valued  
process, and for every $q>0$ and $n\in\{0,1,...,m\}$
\begin{equation}                 \label{estimate of main theorem}
E\sup_{t\in[0,T]}|u_t|_{W^n_p}^q
\leq N(E|\psi|^q_{W^n_p}+E\mathcal K^q_{n,p}(T)) 
\end{equation}
with $N=N(m,p,q,d, M,K,T)$. 
\end{theorem}

In the case $p=2$ we present also a modification of 
Assumption \ref{assumption systemparabolicity}, 
in order to cover   
an important class of stochastic PDE systems, 
the {\it hyperbolic symmetric systems}.

Observe that if in \eqref{semidefinite 1} we replace
$\beta^{ikl}$ with $\beta^{ilk}$, nothing will change.
By the convexity of $t^{2}$ condition \eqref{semidefinite 1}
then holds if we repace $\beta^{ilk}$ with $(1/2)[
 \beta^{ilk}+\beta^{ikl}]$.
Since
$$
|a-b|^{2}\leq |a+b|^{2}+2a^{2}+2b^{2}
$$
this implies that
\eqref{semidefinite 1} also holds for 
$$
\bar{\beta}^{ikl}=(\beta^{ikl}-\beta^{ilk})/2
$$
in place of $\beta^{ikl}$,
which is the antisymmetric part of $\beta^i=b^i-\sigma^{ir}\nu^{r}$.

Hence the following condition is weaker than 
Assumption \ref{assumption systemparabolicity}.

\begin{assumption}                \label{assumption antisymmetric}
There exist a constant 
$K_0>0$ and a $\cP\times\mathcal B(\bR^d)$-measurable 
$\bR^M$-valued function $h=(h^{i}_t(x))$ such that \eqref{h} 
holds, and 
for all  $\omega\in\Omega$, $t\geq0$ and $x\in\bR^d$ 
and for all $(\lambda_1,...,\lambda_d)\in\bR^d$  
\begin{equation}                                                                            \label{semidefinite 2}                        
|\sum_{i=1}^d({\bar\beta}^{ikl}-\delta^{kl}h^i)\lambda_i|^2
\leq K_0\sum_{i,j=1}^d\alpha^{ij}\lambda_i\lambda_j\quad\text{for $k,l=1,...,M$.}
\end{equation}
\end{assumption}

The following result in the special case of deterministic 
PDE systems is indicated and a proof is sketched in 
\cite{GK2006}. 

\begin{theorem}                                                                                \label{theorem antisymmetric}
 Let Assumptions 
 \ref{assumption systemcoeff},  
 \ref{assumption systemdata} 
 and 
 \ref{assumption antisymmetric} 
 hold with $m\geq1$ and with $p=2$. 
Then the conclusion of Theorem \ref{theorem systemmain} holds with 
$p=2$. 
\end{theorem}

\begin{remark}
Notice that Assumption \ref{assumption antisymmetric} obviously
 holds with $h^{i}=0$ if the matrices 
$\beta^i$ are symmetric and $\alpha\geq0$.  When $a=0$ and $\sigma=0$ 
then the system is called a {\it first order symmetric hyperbolic system}.  
\end{remark}

\mysection{Preliminaries}

 First we discuss the solvability of 
\eqref{system}-\eqref{systemini} 
under the strong stochastic parabolicity condition. 

\begin{assumption}                                                                   \label{assumption strongparabolicity}
There is a constant $\kappa>0$ such that 
$$
\alpha^{ij}\lambda_i\lambda_j\geq\kappa\sum_{i=1}^d|\lambda_i|^2
$$
for all $\omega\in\Omega$, $t\geq0$, $x\in\bR^d$ and $(\lambda_1,...,\lambda_d)\in\bR^d$. 
\end{assumption} 
If the above non-degeneracy assumption holds then we need weaker regularity conditions 
on the coefficients and the data than in the degenerate case. 
Recall that $m\geq0$ and make the following assumptions. 

\begin{assumption}                                                 \label{assumption nondegsystemcoeff}
The derivatives in $x\in\bR^d$ of $a^{ij}$ up to order $\max(m,1)$ and of $b^{i}$ and $c$ 
up to order $m$ are $\mathcal P\otimes\mathcal B(\bR^d)$-measurable 
functions, bounded by $K$ for all $i,j\in\{1,2,...,d\}$. The derivatives in $x$ of 
the $l_2$-valued functions $\sigma^{i}$ and 
$ l_2(\bT^{M})$-valued 
function $\nu$ up to order $m$ are 
$\mathcal P\otimes\mathcal B(\bR^d)$-measurable $l_2$-valued and 
$l_2(\bT^{M})$-valued functions, respectively, in magnitude bounded by $K$.  
\end{assumption}
\begin{assumption}                                                 \label{assumption nondegsystemdata}
The free data, $(f_t)_{t\in[0,T]}$ and $(g_t)_{t\in[0,T]}$ are 
predictable 
processes with values in $ W^{m-1}_2(\bR^{d},\bR^M)$ 
and $ W^m_2(\bR^{d},l_2(\bT^{M}))$, respectively, 
such that almost surely
$$
\mathcal K_{m-1,2}^2(T)=\int_0^T
\big(|f_t|^2_{W^{m-1}_2}+|g_t|^2_{W^{m}_2}\big)\,dt<\infty.
$$
The initial value, $\psi$ is an $\cF_0$-measurable 
random variable with values in $ 
W^m_2(\bR^{d},\bR^M)$. 
\end{assumption}

The following is a standard result from the $L_2$-theory of stochastic PDEs. 

\begin{theorem}                                                        \label{theorem L_2}
Let Assumptions \ref{assumption strongparabolicity}, 
 \ref{assumption nondegsystemcoeff} and \ref{assumption nondegsystemdata} 
hold with $m\geq0$. Then \eqref{system}-\eqref{systemini} has a unique solution 
$u$. Moreover, $u$ is a continuous $W^m_2(\bR^{d},\bR^M)$-valued process such that 
$u_t\in W^{m+1}(\bR^{d},\bR^M)$ for $P\times dt$ everyhere, and 
$$                                                                    
E\sup_{t\in[0,T]}|u_t|^2_{W^m_2}
+E\int_0^T|u_t|^2_{W^{m+1}_2}\,dt
$$
\begin{equation}                                         \label{nondegsystemestimate}
\leq N(E|\psi|^2_{W^m_2}
+E\int_0^T|f_t|^2_{W^{m-1}_2}+|g_t|^2_{W^m_2}\,dt) 
\end{equation}
with $N=N(\kappa,m,d, M,K,T)$. 
\end{theorem}
 
The crucial step in the proof of Theorem \ref{theorem main} 
is to obtain an apriori estimate, like estimate \eqref{estimate}. 
In order to discuss the way how such estimate can be proved, take $q=p$,
$M=1$, 
and 
for simplicity assume that $(a^{ij})$ is nonnegative definite, it is bounded and has bounded derivatives up to a sufficiently high order, and that all the other coefficients 
and free terms in equation \eqref{equation} are equal to zero. Thus we consider now the 
PDE
\begin{equation}                                                                    \label{simpleequation}
du(t,x)=a^{ij}(t,x)D_{ij}u(t,x)\,dt, \quad t\in[0,T],\quad x\in\bR^d, 
\end{equation}
with initial condition \eqref{ini}, where we assume that  
$\psi$ is a smooth function from $W^1_p$. We want to obtain the estimate 
\begin{equation}                                                                     \label{simpleestimate}
|u(t)|_{W^1_p}^p\leq N|\psi|^p_{W^1_p}
\end{equation}
for smooth solutions $u$ to \eqref{simpleequation}-\eqref{ini}. 
  
After applying $D_k$ to both sides of equation \eqref{simpleequation} and writing 
$v_k$ in place of $D_kv$, by the chain rule we have 
$$
d\sum_k|u_k|^{p}=p|u_k|^{p-1}u_k(a_k^{ij}u_{ij}+a^{ij}u_{ijk})\,dt. 
$$
Integrating over $\mathbb R^d$  we get 
$$
d\sum_k|u_k|^{p}_{L_p}=\int_{\mathbb R^d}Q(u)\,dx\,dt, 
$$
where 
$$
Q(u)=p|u_k|^{p-2}u_k(a^{ij}u_{ijk}+a_k^{ij}u_{ij}). 
$$
To obtain  \eqref{simpleestimate} we want to have the estimate  
\begin{equation}                                                                          \label{Q}
\int_{\mathbb R^d}Q(v)\,dx\leq N||v||^p_{W^1_p}
\end{equation}
for any smooth $v$ with compact support. To prove this we write 
$
\xi\sim\eta
$
if $\xi$ and $\eta$ have identical integrals over $\mathbb R^d$  and 
we write $\xi \preceq\eta$ if $\xi\sim\eta+\zeta$ such that 
$$
\zeta\leq N(|v|^p+|Dv|^p). 
$$
Then by integration by parts and standard estimates we have 
$$
|v_k|^{p-2}v_ka^{ij}v_{ijk}\preceq -(p-1)|v_k|^{p-2}a^{ij}v_{ki}v_{kj}. 
$$
By the simple inequality 
$\alpha\beta\leq \varepsilon^{-1} \alpha^2+ \varepsilon \beta^2$ 
we have 
$$
|v_k|^{p-2}v_ka_k^{ij}v_{ij}
\leq \varepsilon^{-1}|v_k|^p+\varepsilon |a_k^{ij}v_{ij}|^2. 
$$
for any $\varepsilon>0$. To estimate the term $ |a_k^{ij}v_{ij}|^2$
we use the following lemma, which is 
well-known from \cite{OR2}. 

\begin{lemma}                                   \label{lemma derivative}
Let $a=(a^{ij}(x))$ be a function defined on $\bR^d$, with values 
in the set of non-negative $m\times m$ matrices, such that $a$ and its derivatives in $x$ up 
second order are bounded in magnitude by a constant $K$. 
Let $V$ be a symmetric $m\times m$ matrix. 
Then 
$$
|Da^{ij}V^{ij}|^2
\leq Na^{ij}V^{ik}V^{jk}
$$
for 
 every $x\in\bR^d$, where $N$ is a constant depending only on 
$K$ and $d$. 
\end{lemma}
By this lemma 
$
|a_k^{ij}v_{ij}|^2\leq N a^{ij}v_{il}v_{jl}. 
$
Hence 
$$
|v_k|^{p-2}v_ka_k^{ij}v_{ij}
\preceq N\varepsilon |v_k|^{p-2}a^{ij}v_{il}v_{jl}.  
$$
Thus for each fixed $k=1,2,...,d$ we have 
\begin{equation}                                                                            \label{Qproblem}
Q(v)\preceq -p(p-1)|v_k|^{p-2}a^{ij}v_{ki}v_{kj}
+\varepsilon |v_k|^{p-2}a^{ij}v_{il}v_{jl}
\end{equation}
for any $\varepsilon>0$. Notice that for each fixed $k$ there is 
a summation with respect to $l$ over $\{1,2,...,d\}$ in the expression 
$\varepsilon|v_k|^{p-2}a^{ij}v_{il}v_{jl}$, and   
terms with $l\neq k$ cannot be killed by the expression  
\begin{equation}                                                                           \label{term}
-p(p-1)|v_k|^{p-2}a^{ij}v_{ki}v_{kj}.
\end{equation}
Hence we can get \eqref{Q} when $d=1$ or $p=2$, but we does not get it 
for $p>2$ and $d>1$. To cancel every term in the sum 
$\varepsilon|v_k|^{p-2}a^{ij}v_{il}v_{jl}$ we need an expression like 
$$
-\nu|v_k|^{p-2}a^{ij}v_{li}v_{lj},  
$$
with a constant $\nu$, in place of \eqref{term}, 
for each $k\in\{1,..,d\}$ in the right-hand side of \eqref{Qproblem}. 
This suggests to get \eqref{simpleestimate} via an equation for 
$|\,|Du|^2|_{L_{p/2}}^{p/2}$
instead of that for $\sum_k |D_ku|_{L_p}^p$. 

Let us test this idea. 
From 
$$
du_{k}=(a^{ij}u_{ijk}+a^{ij}_{k}u_{ij})\,dt 
$$
by the chain rule we have 
$$
d|Du|^{2}=2u_{k}a^{ij}u_{ijk}\,dt+2u_{k}a^{ij}_{k}u_{ij}\,dt
\leq a^{ij}[|Du|^{2}]_{ij}\,dt-2a^{ij}u^{ik}u_{jk}\,dt
$$
$$
+N|Du|[a^{ij}u_{ik}u_{jk}]^{1/2}\,dt
\leq a^{ij}[|Du|^{2}]_{ij}\,dt +N|Du|^{2}\,dt
$$
with a constant $N$. Hence 
$$
d(|Du|^{2})^{p/2}\leq(p/2)|Du|^{p-2}a^{ij}[|Du|^{2}]_{ij}\,dt
 +N|Du|^{p}\,dt
$$
$$
\preceq - (p/2)|Du|^{p-2}a^{ij}_{j}[|Du|^{2}]_{i}\,dt+N|Du|^{p}\,dt
$$
$$
=-a^{ij}_{j}[|Du|^{p}]_{i}\,dt+N|Du|^{p}\,dt\preceq  N|Du|^{p}\,dt, 
$$
which implies 
$$
|\,|Du|^2|_{L_{p/2}}^{p/2}\leq N|\,|D\psi|^2| _{L_{p/2}}^{p/2}, 
$$
by Gronwall's lemma. Consequently,  estimate \eqref{simpleestimate} follows, 
since it is not difficult to 
see that 
$$
|u(t)|^p_{L_p}\leq N|\psi|^p_{L_p}
$$
holds.

The following lemma on It\^o's formula in the special case $M=1$ 
is Theorem 2.1 from \cite{K}. The proof of this multidimensional variant goes the same way, and therefore will be omitted. Note that for $p\geq 2$ the second derivative, 
$D_{ij}\langle x\rangle^p$ of the function $(x_1,x_2,\ldots,x_M)\rightarrow\langle x\rangle^p$ 
for $p\geq2$ is
$$
p(p-2)\langle x\rangle^{p-4}x_ix_j+p\langle x\rangle^{p-2}\delta_{ij}, 
$$
which makes the last term in \eqref{itoformula lp} below
 natural. Here and later on we use the convention $0\cdot 0^{-1}:=0$ whenever such terms occur.

\begin{lemma}												\label{lemma itoformula}
Let $p\geq2$ and let  $\psi=(\psi^k)_{k=1}^M$ be an $L_p(\bR^{d},
\bR^M)$-valued 
$\cF_0$-measurable random variable. For $ i=0,1,2,...,d$ and 
$k=1,...,M$ 
let $f^{ki}$ and $(g^{kr})_{r=1}^{\infty}$ be predictable functions 
on $\Omega\times(0,T]$, with values in $L_p$ and in $L_p(l_2)$, respectively, 
such that 
$$
\int_0^T\big(\sum_{i,k}|f^{ki}_t|_{L_{p}}^p+
\sum_{k}|g^{k\cdot}_t|_{L_{p}}^p\big)\,dt<\infty\quad\text{(a.s.)}. 
$$
Suppose that 
for each $k=1,...,M$ we are given a $W^1_p$-valued predictable function 
 $u^k$ on $\Omega\times(0,T]$ such that 
 $$
 \int_0^T|u^k_t|^p_{W^1_p}\,dt<\infty\,(a.s.), 
 $$
and  for any $\phi\in C_0^{\infty}$ 
 with probability 1 for all $t\in[0,T]$ we have
$$
(u_t^k,\phi)=(\psi^k,\phi)+
\int_0^t(g^{kr}_s,\phi)\,dw^r_s
+\int_0^t((f^{k0}_s,\phi)-(f^{ki}_s,D_i\phi))\,ds.
$$
Then there exists a set $\Omega'\subset\Omega$ of full probability 
such that 
$$
u={\mathbf 1}_{\Omega^{\prime}}(u^1,...,u^k)_{t\in[0,T]}
$$  
  is a continuous $L_p(\bR^{d},\bR^{M})$-valued process, and for all $t\in[0,T]$
$$
\int_{\bR^d}\langle u_t\rangle^p\,dx
=\int_{\bR^d}\langle \psi\rangle^p\,dx
+\int_0^t\int_{\bR^d}p\langle u_s\rangle^{p-2}\langle 
u_s,g_s^{  r}\rangle\,dx\,dw_s^{ r}
$$
$$
+\int_0^t\int_{\bR^d}\Big(p\langle u_s\rangle^{p-2}\langle u_s, f_s^0\rangle-p\langle u_s\rangle^{p-2}\langle D_i u_s,f_s^i\rangle
$$
$$
-(1/2)p(p-2)\langle u_s\rangle^{p-4}
\langle u_s,f^i_s\rangle D_i\langle u_s\rangle^2
$$
\begin{equation}														\label{itoformula lp}
+\sum_r\big[(1/2)p(p-2)\langle u_s\rangle^{p-4}\langle
 u_s,g_s^{ r}\rangle^2
+(1/2)p\langle u_s\rangle^{p-2}\langle 
g_s^{  r}\rangle^2\big]\Big)\,dx\,ds, 
\end{equation}
where $f^{i}:=(f^{ki})_{k=1}^{M}$ and $g^{  r}
:=(g^{kr})_{k=1}^M$ 
for all $i=0,1,...,d$ and $r=1,2,...$.
\end{lemma}

\mysection{The main estimate}

Here we consider the problem 
\eqref{system}-\eqref{systemini} with 
$a_{t}=(a^{ij}_{t}(x))$ taking values  in the set
of nonnegative symmetric $d\times d$
matrices and the other  coefficients and  the data
 are described in \eqref{4.12.1}. We also assume that
on $\Omega\times(0,\infty)\times\bR^{d}$
we are given an $\bR^{d}$-valued 
function $h_{t}(x)$.

\begin{lemma}
                                                                                                 \label{lemma 12.17.1}
Suppose that Assumptions \ref{assumption systemcoeff},
\ref{assumption systemdata}, 
and \ref{assumption systemparabolicity} hold with $m\geq0$. 
Assume that  
$u=(u_t)_{t\in[0,T]}$ is a solution of 
 \eqref{system}-\eqref{systemini}   
on $[0,T]$ (as defined before
Theorem \ref{theorem systemmain}). Then (a.s.) $u$ is a 
continuous $L_p(\bR^{d},\bR^{M})$-valued process, 
and there is a constant $N=N(p,K,d,M,K_0)$ such that
$$
d\int_{\bR^{d}}\langle u_{t}\rangle^{p} \,dx
+(p/4)\int_{\bR^{d}}\langle  u_{t}\rangle^{p-2}
\alpha^{ij}_{t}\langle D_{i}u_{t},D_{j}u_{t}\rangle \,dx\,dt
$$
$$
\leq p\int_{\bR^{d}}\langle u_{t}\rangle^{p-2} 
\langle u_{t},\sigma^{ik}D_{i}u_{t}+
\nu^{k}_{t}u_{t}+g^{k}_{t}\rangle
\,dx\,dw^{k}_{t}
$$
\begin{equation}
                                              \label{12.17.3}
+N\int_{\bR^{d}}\big [\langle u_{t}\rangle^{p}
+\langle f_{t}\rangle^{p} 
+\big(\sum_{k}\langle g^{k}_{t}\rangle^{2 }\big)^{p/2}
 +\big(\sum_{k}\langle D g^{k}_{t}\rangle^{2 }
\big)^{p/2} \big]\,dx\,dt.  
\end{equation}
\end{lemma}

\begin{proof} 
By Lemma \ref{lemma itoformula} (a.s.) $u$ is a 
continuous $L_p(\bR^{d},\bR^{M})$-valued process 
and 
$$
d\int_{\bR^d}\langle u_{t}\rangle^{p}\,dx=
\int_{\bR^d}p\langle u_t\rangle^{p-2}\langle u_t,\sigma^{ik}D_iu_t+
\nu^k_tu_t+g^k_t\rangle\,dx\,dw^k_t
$$
$$
+\int_{\bR^d}\Big(p\langle u_t\rangle^{p-2}
\langle u_t,b^i_t D_iu_t+c_tu_t+f_t-D_ia^{ij}_tD_ju_t\rangle-
p\langle u_t\rangle^{p-2}\langle D_i u_t,a^{ij}_t D_ju_t\rangle
$$
$$
-(1/2)p(p-2)\langle u_t\rangle^{p-4}
D_i\langle u_t\rangle^2\langle u_t,a^{ij}_tD_ju_t\rangle
$$
$$
+\sum_k\big\{(1/2)p(p-2)\langle u_t\rangle^{p-4}
\langle u_t,\sigma^{ik}_tD_i u_t+\nu^k_tu_t+g^k_t\rangle^2
$$
\begin{equation}										\label{12.17.2}
+(1/2)p\langle u_t\rangle^{p-2}\langle \sigma^{ik}_tD_i u_t
+\nu^k_tu_t+g^k_t\rangle^2\big\}\Big)\,dx\,dt.
\end{equation}
 
Observe that
$$
\langle u_{t}\rangle^{p-2}
\langle u_{t},f_{t}\rangle\leq \langle u_{t}\rangle^{p }
+\langle f_{t}\rangle^{p },\quad
\langle u_{t}\rangle^{p-2}
\sum_{k}\langle g^{k}_{t}\rangle^{2 }
\leq\langle u_{t}\rangle^{p }+
\big(\sum_{k}\langle g^{k}_{t}\rangle^{2 }\big)^{p/2},
$$

$$
\langle u_{t}\rangle^{p-2}
\sum_{k}\langle \nu^{k}_{t}u_{t}, g^{k}_{t}\rangle 
\leq N\langle u_{t}\rangle^{p-1}
\big(\sum_{k}\langle g^{k}_{t}\rangle^{2 }\big)^{1/2}
\leq N\langle u_{t}\rangle^{p }+ N
\big(\sum_{k}\langle g^{k}_{t}\rangle^{2 }\big)^{p/2},
$$

$$
\langle u_{t}\rangle^{p-4}\sum_{k}\langle u_{t},
g^{k}_{t}\rangle^{ 2}
\leq\langle u_{t}\rangle^{p-2}
\sum_{k}\langle g^{k}_{t}\rangle^{2 } 
\leq\langle u_{t}\rangle^{p }+
\big(\sum_{k}\langle g^{k}_{t}\rangle^{2 }\big)^{p/2},
$$

$$
\langle u_{t}\rangle^{p-4}\sum_{k}\langle u_{t},
\nu^{k}_{t}u_{t}\rangle
\langle u_{t},
g^{k}_{t}\rangle\leq N\langle u_{t}\rangle^{p-1}
\big(\sum_{k}\langle g^{k}_{t}\rangle^{2 }\big)^{1/2}
\leq\langle u_{t}\rangle^{p }+
\big(\sum_{k}\langle g^{k}_{t}\rangle^{2 }\big)^{p/2}, 
$$

$$
\langle u_t\rangle^{p-2}\langle u_t,c_tu_t\rangle
\leq \langle u_t\rangle^{p-1}\langle c_tu_t\rangle
\leq |c_t|\langle u_t\rangle^{p}, 
$$
where $|c|$ denotes the (Hilbert-Schmidt) norm of $c$.  

This shows how to estimate a few terms on the right
in \eqref{12.17.2}. We write 
$
\xi\sim\eta
$
if $\xi$ and $\eta$ have identical integrals over $\bR^d$  and 
we write $\xi \preceq\eta$ if $\xi\sim\eta+\zeta$ and
the integral of $\zeta$ over $\bR^d$ can be estimated by 
the coefficient of $dt$ in the right-hand side of 
\eqref{12.17.3}. For instance, integrating by parts
and using the smoothness of $\sigma^{ik}_{t}$ 
and $g^{k}_{t}$ we get
\begin{equation}                                                                   \label{11.14.3.14}
p\langle u_{t}\rangle^{p-2}
\langle \sigma^{ik}_{t}D_{i}u_{t},g^{k}_{t}\rangle 
 \preceq  -p\sigma^{ik}_{t}(D_{i}\langle u_{t}\rangle^{p-2})
\langle u_{t},g^{k}_{t}\rangle
\end{equation}
$$
=-p(p-2)\langle u_{t}\rangle^{p-4}
\langle u_{t},\sigma^{ik}_{t}D_{i}u_{t} \rangle 
\langle  u_{t},g^{k}_{t}\rangle ,
$$
where the first expression comes from the  last 
occurence of $g^{k}_{t}$ in \eqref{12.17.2}
and the last one with an opposite sign appears
in the evaluation of the next to last factor of $dt$ 
in \eqref{12.17.2}. 
 Notice, however, that these  
calculations are not justified when 
$p$ is close to $2$, since in this case 
$\langle u_{t}\rangle^{p-2}$ may not be  
absolutely continuous with respect to $x^i$ 
and it is not clear either if $0/0$ should be defined as $0$ 
when it occurs in the second line. For $p=2$ we clearly have 
$
\langle \sigma^{ik}_{t}D_{i}u_{t},g^{k}_{t}\rangle 
 \preceq  0.  
$
For $p>2$ we modify the above calculations by approximating 
the function $\langle t\rangle^{p-2}$, $t\in\bR^M$, by continuously differentiable functions 
$\phi_n(t)=\varphi_n(\langle t\rangle^2)$ 
such that 
$$
\lim_{n\to\infty}\varphi_n(r)=|r|^{(p-2)/2}, 
\quad 
\lim_{n\to\infty}\varphi^{\prime}_n(r)
=(p-2){\rm{sign}}(r)|r|^{(p-4)/2}/2
$$
for all $r\in\bR$, and 
$$
|\varphi_n(r)|\leq N|r|^{(p-2)/2}, 
\quad |\varphi^{\prime}_n(r)|\leq N|r|^{(p-4)/2}
$$ 
for all $r\in\bR$ and integers $n\geq1$, 
where $\varphi^{\prime}_n:=d\varphi_n/dr$ and 
$N$ is a constant independent of $n$. 
Thus 
instead of \eqref{11.14.3.14} 
we have 
\begin{equation}                                                                \label{12.14.3.14}                                                 
p\varphi_n(\langle u_{t}\rangle^2)
\langle \sigma^{ik}_{t}D_iu_{t},g^k_{t}\rangle 
\preceq
-2 p\varphi^{\prime}_n(\langle u_{t}\rangle^2)
\langle u_{t},\sigma^{ik}_{t}D_{i}u_{t} \rangle  
\langle  u_{t},g^{k}_{t}\rangle,  
\end{equation}
where 
\begin{equation}                                                              \label{13.14.3.14} 
|\varphi^{\prime}_n(\langle u_{t}\rangle^2)
\langle u_{t},\sigma^{ik}_{t}D_{i}u_{t} \rangle 
\langle  u_{t},g^{k}_{t}\rangle| 
\leq N\langle u_{t}\rangle^{p-2}
\langle D_{i}u_{t} \rangle 
\langle g^{k}_{t}\rangle 
\end{equation}
with a constant $N$ independent of $n$. 
Letting $n\to\infty$ in \eqref{12.14.3.14} we get 
$$
p\langle u_{t}\rangle^{p-2}
\langle \sigma^{ik}_{t}D_{i}u_{t},g^{k}_{t}\rangle 
 \preceq  -p(p-2)\langle u_{t}\rangle^{p-4}
\langle u_{t},\sigma^{ik}_{t}D_{i}u_{t} \rangle 
\langle  u_{t},g^{k}_{t}\rangle,  
$$ 
where, due to \eqref{13.14.3.14}, $0/0$ means $0$ when it occurs  .

These manipulations allow us to take care of the terms
containing $f$ and $g$ and
 show that to prove the lemma we have to prove that
$$
p( I_{0}+I_{1} +I_{2})+(p/2)I_{3}+[p(p-2)/2]( I_{4}+I_{5} )
$$
\begin{equation}
                                             \label{12.17.4}
\preceq-(p/4)\langle  u_{t}\rangle^{p-2}
\alpha^{ij}_{t}\langle D_{i}u_{t},D_{j}u_{t}\rangle ,
\end{equation}
where
$$
 I_{0}=-\langle u_t\rangle^{p-2}D_ia^{ij}_t\langle
u_t,D_ju_t\rangle ,
\quad  I_{1}=-\langle u_t\rangle^{p-2}a^{ij}_t\langle D_i u_t,
D_ju_t\rangle 
$$
$$
I_{2}=\langle u_{t}\rangle^{p-2}
\langle u_{t},b^{i}_{t}D_{i} u_{t}  \rangle,\quad
I_{3}=\langle u_{t}\rangle^{p-2}
\sum_{k}\langle \sigma^{ik}_{t}D_{i}u_{t}+\nu^{k}_{t}u_{t}
 \rangle^{2},
$$

$$
I_{4}=
\langle u_{t}\rangle^{p-4}\sum_{k}
\langle u_{t},\sigma^{ik}_{t}D_{i}u_{t}+\nu^{k}_{t}u_{t}
 \rangle^{2},\quad  I_5=-\langle u_t\rangle^{p-4}D_i\langle
u_t\rangle^2\langle u_t,a^{ij}_tD_ju_t\rangle 
$$

Observe that

$$
I_{0}=-(1/2)\langle u_t\rangle^{p-2}D_ia^{ij}_tD_j\langle u_t\rangle^2=-(1/p)D_j\langle u_t\rangle^pD_ia^{ij}_t\preceq0,
$$
by the smoothness of $a$.  
Also notice that
$$
I_{3} \preceq   \langle  u_{t}\rangle^{p-2}\sigma^{ik}_{t}
\sigma^{jk}_{t}\langle D_{i}u_{t},D_{j}u_{t}\rangle
+I_{6},
$$
where
$$
I_{6}=2\langle u_{t}\rangle^{p-2}
\sigma^{ik}_{t}\langle D_{i}u_{t},\nu^{k}u_{t}\rangle.
$$

It follows that
 
$$
pI_{1}+(p/2)I_{3} \preceq  -(p/2)\langle  u_{t}\rangle^{p-2}
\alpha^{ij}_{t}\langle D_{i}u_{t},D_{j}u_{t}\rangle+(p/2)I_{6}.
$$
 
Next,
$$
I_{4} \preceq  \langle u_{t}\rangle^{p-4}
\sigma^{ik}_{t}\sigma^{jk}_{t}\langle  u_{t},D_{i}u_{t} \rangle
\langle  u_{t},D_{j}u_{t} \rangle
+2\langle u_{t}\rangle^{p-4}
\sigma^{ik}_{t}\langle u_{t},D_{i}u_{t}\rangle
\langle u_{t},\nu^{k}_{t}u_{t}\rangle
$$

$$
=(1/4)\langle u_{t}\rangle^{p-4}
\sigma^{ik}_{t}\sigma^{jk}_{t}
D_{i}\langle  u_{t} \rangle^{2}
D_{j}\langle  u_{t} \rangle^{2}
+[2/(p-2)](D_{i}\langle  u_{t} \rangle^{p-2})
\sigma^{ik}_{t}\langle u_{t},\nu^{k}_{t}u_{t}\rangle 
$$

$$
\preceq (1/4)\langle u_{t}\rangle^{p-4}
\sigma^{ik}_{t}\sigma^{jk}_{t}
D_{i}\langle  u_{t} \rangle^{2}
D_{j}\langle  u_{t} \rangle^{2}-
[1/(p-2)]I_{6}-[2/(p-2)]I_{7},
$$
where
$$
I_{7}= \langle  u_{t} \rangle^{p-2} 
\sigma^{ik}_{t}\langle u_{t},\nu^{k}_{t}D_{i}u_{t}\rangle.
$$
Hence
$$
pI_{1}+(p/2)I_{3}+[p(p-2)/2](I_{4} +I_{5} )
 \preceq 
-(p/2)\langle  u_{t}\rangle^{p-2}
\alpha^{ij}_{t}\langle D_{i}u_{t},D_{j}u_{t}\rangle
$$

$$
-[p(p-2)/ 8 ]
\langle u_{t}\rangle^{p-4}\alpha^{ij}_{t}
D_{i}\langle  u_{t} \rangle^{2}
D_{j}\langle  u_{t} \rangle^{2}-pI_{7}, 
$$
 and 
$$
I_2-I_7=\langle  u_{t} \rangle^{p-2}( \langle u_{t},b^{i}_{t}
D_{i} u_{t}\rangle
-\sigma^{ik}_{t}\langle u_{t},\nu^{k}_{t}D_{i}u_{t}\rangle)=
\langle  u_{t} \rangle^{p-2}\langle u_{t},\beta^{i}_{t}D_{i} u_{t}\rangle, 
$$
with $\beta^i=b^i-\sigma^{ik}\nu^{k}$. 
It follows by Remark \ref{remark reformulation} 
that the left-hand side of \eqref{12.17.4} is estimated
in the order defined by $\preceq$ by
$$
-(p/2)\langle  u_{t}\rangle^{p-2}
\alpha^{ij}_{t}\langle D_{i}u_{t},D_{j}u_{t}\rangle
$$
 $$
-[p(p-2)/ 8 ]
\langle u_{t}\rangle^{p-4}\alpha^{ij}_{t}
D_{i}\langle  u_{t} \rangle^{2}
D_{j}\langle  u_{t} \rangle^{2}
$$

$$
+K_{0}p
\langle  u_{t}\rangle^{p-2}\big|\sum_{i,j=1}^{d}
\alpha^{ij}_{t}\langle
D_{i}u_{t},D_{j}u_{t}\rangle
\big|^{1/2}\langle u_{t}\rangle +h^{i}D_{i}\langle  u_{t}\rangle^{p}
$$

\begin{equation}
                                                                                                                  \label{12.18.1}
\preceq-(p/4)\langle  u_{t}\rangle^{p-2}
\alpha^{ij}_{t}\langle D_{i}u_{t},D_{j}u_{t}\rangle 
-[p(p-2)/ 8 ]
\langle u_{t}\rangle^{p-4} \alpha ^{ij}_{t}
D_{i}\langle  u_{t} \rangle^{2}
D_{j}\langle  u_{t} \rangle^{2}
\rangle,
\end{equation}
where the last relation follows from the 
elementary  inequality
$ab\leq\varepsilon a^{2}+\varepsilon^{-1}b^{2}$. The lemma is proved.
\end{proof}

\begin{remark}                                \label{remark 1.23.3.14}
In the case that $p=2$ one can replace
condition \eqref{semidefinite 1}
with the following:

  There are  constant $K_{0},N\geq0$ 
such that
for all continuously differentiable $\bR^{M}$-valued
functions $u=u(x)$ with compact support in $\bR^{d}$ 
and all values of the arguments we have
$$                                      
 \int_{\bR^d}\langle  u, \beta^{i}D_{i}u\rangle\,dx
 \leq N\int_{\bR^d}\langle u\rangle^2\,dx
 $$
 \begin{equation}                            \label{2.23.3.14}
+K_{0}\int_{\bR^d}\big
(\big|\sum_{i,j=1}^{d}\alpha^{ij}\langle
D_{i}u,D_{j}u\rangle
\big|^{1/2}\langle u\rangle +h^{i} 
\langle D_{i} u,u \rangle\big)\,dx.
\end{equation}
This condition is weaker than \eqref{semidefinite 1} 
as follows from Remark \ref{remark reformulation}
and still by inspecting the above proof
we get that  $u$ is a continuous $L_2(\bR^{d},\bR^{M})$-valued process, 
and there is a constant $N=N(K,d,M,K_0)$ such that 
\eqref{12.17.3} holds with $p=2$. 
\end{remark}

\begin{remark}                              \label{remark 1.22.3.14}
In the case that $p=2$  and the magnitudes of the
 first derivatives of $b^{i}$ are
bounded by $K$ one can further replace condition \eqref{2.23.3.14}
with a more tractable one, which is Assumption
\ref{assumption antisymmetric}.

  Indeed, for $\varepsilon>0$ 
$$
R:= \langle u,(\beta^i-h^iI_{M})D_{i} u\rangle
=\tfrac{1}{2}\beta^{ikl}D_i(u^ku^l)+\langle u,(\bar{\beta}^i-h^iI_M)D_{i}
u\rangle
$$
$$
\leq \tfrac{1}{2}\beta^{ikl}D_i(u^ku^l)+
\varepsilon\langle(\bar{\beta}^i-h^iI_M)D_{i}
u\rangle^2/2+\varepsilon^{-1}\langle u\rangle^2/2. 
$$ 
Using Assumption \ref{assumption antisymmetric} we  get 

$$ R\leq \tfrac{1}{2}\beta^{ikl}D_i(u^ku^l) +\varepsilon
MK_0\alpha^{ij}\langle D_iu, D_ju\rangle/2+\varepsilon^{-1}\langle
u\rangle^2/2
$$ 
for every $\varepsilon>0$. Hence by integration by parts we have 
$$
\int_{\bR^d}\langle u,\beta^iD_{i} u\rangle\,dx\leq N\int_{\bR^d}\langle u\rangle^2\,dx
+\int_{\bR^d}\langle u,h^iI_{M}D_{i} u\rangle\,dx
$$
$$
+
MK_0\int_{\bR^d}
(\varepsilon/2) \alpha^{ij}\langle D_iu_t, D_ju_t\rangle
+(\varepsilon^{-1}/2)\langle u\rangle^2\,dx. 
$$
Minimising here over $\varepsilon>0$ we get \eqref{2.23.3.14}.  
In that case again $u$ is a continuous $L_2(\bR^{d},\bR^{M})$-valued process, 
and there is a constant $N=N(K,d,M,K_0)$ such that 
\eqref{12.17.3} holds with $p=2$.
\end{remark}

\begin{remark}
                                                    \label{remark 12.19.1}
If $M=1$, then condition \eqref{12.19.1} is 
obviously satisfied 
with $K_{0}=0$ and $h^{i}=b^{i}-\sigma^{ik}\nu^{k}$.

Also note that in the general case, if the coefficients are
smoother, then by {\em formally\/} differentiating
equation \eqref{system} with respect to $x^{i}$ we obtain
a new system of equations for the $M\times d$
matrix-valued function 
$$
v_{t}=(v_{t}^{nm})=Du_{t}=(D_{m}u^{n}_{t}).
$$  
We treat the space of $M\times d$ matrices as a Euclidean 
$Md$-dimensional space, 
the coordinates in which are organized in a special way.
The inner product in this space is then just 
$\langle\langle A,B\rangle\rangle={\rm{tr}}AB^{*}$.
Naturally, linear operators in this space will be given
 by  matrices like $(T^{(nm)(pj)})$, which transforms
an $M\times d$ matrix $(A^{pj})$ into 
an $M\times d$ matrix $(B^{nm})$ by the formula
$$
B^{nm}=\sum_{p=1}^{m}\sum_{j=1}^{d}T^{(nm)(pj)}A^{pj}.
$$

We claim  that the system for $v_{t}$ satisfies
Assumptions \ref{assumption systemcoeff},
\ref{assumption systemdata}, 
and \ref{assumption systemparabolicity} with $m\geq0$
if Assumptions \ref{assumption systemcoeff},
\ref{assumption systemdata}, 
and \ref{assumption systemparabolicity} are satisfied with $m\geq1$.

Indeed, as is easy to see,
$v_{t}$ satisfies \eqref{system} with
the same $\sigma$ and $a$ and with $\tilde b^{i}$, $\tilde c$, $\tilde f$,
$\tilde \nu^{k}$,  $\tilde g^{k}$ in place of 
$  b^{i}$, $  c$, $ f$,
$ \nu^{k}$,  $  g^{k}$, respectively, where
\begin{equation}                                                \label{15.14.4.14}
\tilde{b}^{i(nm)(pj)}=
D_{m}a^{ij}\delta^{pn}+b^{inp}\delta^{jm},\quad
\tilde c^{(nm)(pj)}=c^{np}\delta^{mj}+D_mb^{jnp},
\end{equation}
$$
\tilde{f}^{nm}=
D_{m}f^{ n }+u^{r} D_{m}  c^{nr},\quad
\tilde{\nu}^{k(nm)(pj)}=D_{m}\sigma^{jk}\delta^{np}
+\nu^{knp}\delta^{mj}, 
$$

\begin{equation}                                      \label{3.8.3.14}
\tilde{g}^{knm}=D_{m}g^{kn}+u^{r}D_{m}\nu^{knr}.
\end{equation}
Then the left-hand side of the counterpart of
\eqref{12.19.1} for $v$ is 
$$
\sum_{m=1}^{d}K_{m}+\sum_{n=1}^{M}J_{n},
$$
where (no summation with respect to $m$)
$$
K_{m}=v^{nm}b^{inr}D_{i}v^{rm}-
\sigma^{ik}v^{nm}\nu^{knr}D_{i}v^{rm} 
$$
and (no summation with respect to $n$)
$$
J_{n}=v^{nm}D_{m}a^{ij}D_{i}v^{nj}
-\sigma^{ik}v^{nm}D_{m}\sigma^{jk}D_{i}v^{nj}.
$$
Observe that $D_{i}v^{nj}=D_{ij}u^{n}$ implying  that
$$
\sigma^{ik}D_{m}\sigma^{jk}D_{i}v^{nj}
=(1/2)D_{m}(\sigma^{ik}\sigma^{jk})D_{ij}u^{n},
$$
$$
J_{n}=(1/2)v^{nm}D_{m}\alpha^{ij}D_{ij}u^{n}.
$$
By Lemma \ref{lemma derivative} for any $\varepsilon>0$ 
and $n$ (still no summation with respect to $n$)
$$
J_{n}\leq N\varepsilon^{-1}
\langle\langle v\rangle\rangle^{2}
+\varepsilon\alpha^{ij} D_{ik }u^{n}D_{jk }u^{n},
$$
which along with the fact that
$D_{ik }u^{n}=D_{i}v^{nk}$ yields
$$
\sum_{n=1}^{M}J_{n}\leq
N\varepsilon^{-1}\langle\langle v\rangle\rangle^{2}
 +\varepsilon\alpha^{ij}
\langle\langle D_{i}v,D_{j}v\rangle\rangle.
$$
Upon minimizing with respect to $\varepsilon$ we find
$$
\sum_{n=1}^{M}J_{n}\leq
N\big(\sum_{i,j=1}^{d}\alpha^{ij}
\langle\langle D_{i}v,D_{j}v\rangle\rangle\big)^{1/2}
\langle\langle v\rangle\rangle.
$$

Next, by assumption for any $\varepsilon>0$ and $m$
(still no summation with respect to $m$)
$$
K_{m}\leq N\varepsilon^{-1}
\langle\langle v\rangle\rangle^{2}
+
\varepsilon \alpha^{ij}D_{i}v^{rm}D_{j}v^{rm}
+(1/2)h^{i}D_{i}\sum_{r=1}^{M}(v^{rm})^{2}.
$$
We conclude as above that
$$
\sum_{m=1}^{d}K_{m}\leq
N\big(\sum_{i,j=1}^{d}\alpha^{ij}
\langle\langle D_{i}v,D_{j}v\rangle\rangle\big)^{1/2}
\langle\langle v\rangle\rangle
+h^{i}\langle\langle D_{i}v, v\rangle\rangle 
$$
and this proves our claim.

The above calculations show also that the system for $v_{t}$ satisfies
Assumptions \ref{assumption systemcoeff},
\ref{assumption systemdata}, 
and \ref{assumption antisymmetric} with $m\geq0$
if Assumptions \ref{assumption systemcoeff},
\ref{assumption systemdata}, 
and \ref{assumption antisymmetric} are satisfied with $m\geq1$. 
 (Note that due to 
Assumptions \ref{assumption systemcoeff} with $m\geq1$,  
$\tilde b$, given in \eqref{15.14.4.14}, has first order derivatives in 
$x$, which in magnitude are bounded by a constant.)

 Now higher order derivatives of $u$ are
obviously estimated through lower order ones
on the basis of this
remark without any additional computations. 
However, we still need to be sure that we can
differentiate equation \eqref{system}.  
\end{remark}
 
By the help of the above remarks one can easily estimate the moments of the 
$W^n_p$-norms of $u$ using of the following version of 
Gronwall's lemma.

\begin{lemma}                             \label{lemma 1.6.10.12}
Let $y=(y_t)_{t\in[0,T]}$ 
and $F=(F_t)_{t\in[0,T]}$ 
be  
adapted nonnegative stochastic processes 
and let $m=(m_t)_{t\in[0,T]}$ be a 
continuous local martingale  
such that 
\begin{equation}                                                                    \label{1.7.3.14}
dy_t\leq (Ny_t+F_t)\,dt+dm_t
\quad\text{on $[0,T]$}
\end{equation}
\begin{equation}                                                                       \label{2.8.3.14}
d[m]_t\leq (N y^2_t+y^{2(1-\rho)}_tG^{2\rho}_t)\,dt
\quad\text{on $[0,T]$,}                                       
\end{equation}
with some constants $N\geq0$ and $\rho\in[0,1/2]$, and 
a nonnegative adapted stochastic process 
$G=(G_t)_{t\in[0,T]}$, such that  
$$
\int_0^TG_t\,dt<\infty\,(a.s.),
$$
where $[m]$ is the quadratic variation process for  $m$. 
Then for any $q>0$ 
$$
E\sup_{t\leq T}y_t^q \leq CEy_0^q
+CE\left\{\int_0^{T}(F_t+G_t)\,dt\right\}^{q}   
$$
with a constant $C=C(N,q,\rho,T)$. 
\end{lemma}

\begin{proof}
This lemma improves Lemma 3.7 
from \cite{GS}. Its proof goes in the same way as that 
in \cite{GS}, and can be found in \cite{G12}. 
\end{proof}

\begin{lemma}											\label{lemma estimate}
Let $m\geq0$.
Suppose that Assumptions \ref{assumption systemcoeff},
\ref{assumption systemdata}, 
and \ref{assumption systemparabolicity} 
are satisfied
 and assume that  
$u=(u_t)_{t\in[0,T]}$ is a solution of 
\eqref{system}-\eqref{systemini}   
on $[0,T]$ such that (a.s.)
$$
\int_0^T|u_t|^p_{W_p^{m+1}}\,dt<\infty.
$$
Then (a.s.) $u$ is a continuous $W^m_p(\bR^{d},\bR^{M})$-valued process    
and for any $q>0$ 
\begin{equation}                               \label{1.11.3.14}
E\sup_{t\in[0,T]}|u_t|_{W^m_p}^q
\leq N(E|\psi|_{W^m_p}^q+E\mathcal K^q_{m,p}(T)) 
\end{equation}
with a constant $N=N(m,p,q,d,M,K,K_0,T)$. 
If $p=2$ and instead of 
Assumption \ref{assumption systemparabolicity}
Assumption \ref{assumption antisymmetric}
 holds   and (in case $m=0$) the magnitudes of the
 first derivatives of $b^{i}$ are
bounded by $K$,  then 
$u$ is a continuous $W^m_2(\bR^{d},\bR^{M})$-valued process,     
and for any $q>0$ estimate \eqref{1.11.3.14} holds 
(with $p=2$).  
\end{lemma}

\begin{proof}
We are going to prove   the lemma by induction 
on $m$. First let $m=0$ and denote $y_t:=|u_t|^p_{L_p}$. Then by virtue of 
Remark \ref{remark 1.22.3.14}  and Lemma \ref{lemma 12.17.1},  
the process 
$y=(y_t)_{t\in[0,T]}$ is an adapted $L_p$-valued continuous process, 
and 
\eqref{1.7.3.14} holds with  
$$
F_t:=\int_{\bR^{d}}
\big [
\langle f_{t}\rangle^{p} 
+\big(\sum_{k}\langle g^{k}_{t}\rangle^{2 }\big)^{p/2}
 +\big(\sum_{k}\langle D g^{k}_{t}\rangle^{2 }
\big)^{p/2} \big]\,dx, 
$$
$$
m_t:=p\int_0^t\int_{\bR^{d}}\langle u_{s}\rangle^{p-2} 
\langle u_{s},\sigma^{ik}_sD_{i}u_{s}+
\nu^{k}_{s}u_{s}+g^{k}_{s}\rangle
\,dx\,dw^{k}_{s}. 
$$
Notice that 
$$
d[m_t]=p^2\sum_{r=1}^{\infty}\left(\int_{\bR^{d}}\langle u_{t}\rangle^{p-2} 
\langle u_{t},\sigma^{ir}_tD_{i}u_{t}+
\nu^{r}_{t}u_{t}+g^{r}_{t}\rangle
\,dx\right)^2\,dt. 
$$
$$
\leq 3p^2(A_t+B_t+C_t)\,dt, 
$$
with 
$$
A_t=\sum_{r=1}^{\infty}\left(p\int_{\bR^{d}}\langle u_{t}\rangle^{p-2} 
\sigma^{ir}_t\langle u_{t},D_{i}u_{t}
\rangle\,dx\right)^2
=\sum_{r=1}^{\infty}\left(\int_{\bR^{d}}\sigma^{ir}_tD_i\langle u_{t}\rangle^{p} 
\,dx\right)^2, 
$$
$$
B_t=\sum_{r=1}^{\infty}\left(\int_{\bR^{d}}\langle u_{t}\rangle^{p-2} 
\langle u_{t},
\nu^{r}_{t}u_{t}\rangle
\,dx\right)^2,
\quad 
C_t=\sum_{r=1}^{\infty}\left(\int_{\bR^{d}}\langle u_{t}\rangle^{p-2} 
\langle u_{t},g^{r}_{t}\rangle
\,dx\right)^2. 
$$
Integrating by parts and then using Minkowski's inequality,  
due to   Assumption \ref{assumption regularitycoeff}, we get  
$A_t\leq N y_t^2$ with a constant $N=N(K,M,d)$. 
Using Minkowski's inequality and taking into account that 
$$
\sum_{r=1}^{\infty}\langle u,\nu^{r}u\rangle^2\leq 
\langle u\rangle^4\sum_{r=1}^{\infty}|\nu^r|^2
\leq N\langle u\rangle^4, 
\quad 
\sum_{r=1}^{\infty}\langle u,g^{r}\rangle^2
\leq \langle u\rangle ^2|g|, 
$$
we obtain 
$$
B_t\leq Ny_t^2, \quad C_t\leq \left(\int_{\bR^d}\langle u_t\rangle^{p-1}|g_t|\,dx\right)^2
\leq |y_t|^{2(p-1)/p}|g_t|_{L_p}^{2}. 
$$
Consequently, condition \eqref{2.8.3.14} holds with $G_t=|g_t|^p_{L_p}$, $\rho=1/p$, 
and 
we get \eqref{1.11.3.14} with $m=0$ by applying Lemma \ref{lemma 1.6.10.12}. 

Let $m\geq1$ and assume that the assertions of the lemma are valid 
 for $m-1$, in place of $m$, 
for any $M\geq1$, $p\geq2$ and $q>0$, for any $u$, $\psi$, $f$ and $g$ 
satisfying the assumptions  
with $m-1$ in place of $m$. Recall the notation $v=(v^{nl}_t)=(D_lu^n_t)$
from  Remark \ref{remark 12.19.1}, and that $v_{t}$ satisfies
\eqref{system} with
the same $\sigma$ and $a$ and with $\tilde b^{i}$, $\tilde c$, $\tilde f$,
$\tilde \nu^{k}$,  $\tilde g^{k}$ in place of 
$  b^{i}$, $  c$, $ f$,
$ \nu^{k}$,  $  g^{k}$, respectively.  
By virtue of Remarks \ref{remark 12.19.1} and \ref{remark 1.22.3.14}  the
system for $v=(v_t)_{t\in[0,T]}$  satisfies Assumption \ref{assumption
systemparabolicity},   and it is easy to see that it satisfies also 
Assumptions \ref{assumption systemcoeff} and
\ref{assumption systemdata} with $m-1$ in place of $m$. Hence 
by the induction hypothesis $v$ is a continuous 
$W^{m-1}_p(\bR^{d},\bR^{M})$-valued adapted process, and  we have 
\begin{equation}                                   \label{1.12.3.14}
E\sup_{t\in[0,T]}|v_t|_{W^{m-1}_p}^q
\leq N(E|{\tilde \psi}|_{W^{m-1}_p}^q
+E\tilde{\mathcal K}^q_{m-1,p}(T)) 
\end{equation}
with a constant $N=N(T,K,K_0,M,d,p,q)$, where ${\tilde\psi}^{nl}=D_{l}\psi^n$, 
$$
\tilde{\mathcal K}^p_{m-1,p}(T)
:=\int_0^T(|\tilde f_t|^p_{W^{m-1}_p}+|\tilde g_t|^p_{W^{m}_p})\,dt. 
$$
It follows that
 $(u_t)_{t\in[0,T]}$ is a $W^m_p(\bR^{d},\bR^{M})$-valued 
continuous adapted process,  and by using the induction hypothesis it is easy to see 
that 
$$
E\tilde{\mathcal K}^q_{m-1,p}(T))\leq 
N(E|\psi|^q_{W^m_p}+E{\mathcal K}^q_{m,p}(T)).  
$$ 
Thus \eqref{1.11.3.14} follows. 

If $p=2$ and Assumption \ref{assumption systemparabolicity}  is replaced
with
 Assumptions \ref{assumption antisymmetric},
 then the proof of the conclusion of the lemma 
goes in the same way  with obvious changes. The proof is complete. 
\end{proof}

\mysection{Proof of Theorems \ref{theorem systemmain} and \ref{theorem antisymmetric}}

First we prove uniqueness.
Let $u^{(1)}$ and $u^{(2)}$ be 
 solutions 
to \eqref{system}-\eqref{systemini}, 
and let Assumptions  \ref{assumption systemcoeff}, \ref{assumption systemdata} 
and \ref{assumption systemparabolicity} hold with $m=0$. 
Then $u:=u^{(1)}-u^{(2)}$ solves 
\eqref{system} with $u_0=0$, $g=0$ and $f=0$ and 
Lemma \ref{lemma 12.17.1} and Remark \ref{remark 1.22.3.14}
are applicable to $u$.
Then using It\^o's formula 
for transforming
 $|u_t|^p_{L_p}\exp(-\lambda t)$ with a 
sufficiently large constant $\lambda$, 
 after simple calculations we get that almost surely 
$$
0\leq e^{-\lambda t}|u_t|^p_{L_p}\leq m_t\quad\text{for all $t\in[0,T]$},  
$$
where $m:=(m_t)_{t\in[0,T]}$ is a continuous local martingale starting from $0$.  
Hence 
almost surely $m_t=0$ for all $t$, and it follows that almost surely $u_t(x)=v_t(x)$ 
for all $t$ and almost every 
$x\in\bR^d$. If 
$p=2$ 
 and  Assumptions  \ref{assumption systemcoeff}, \ref{assumption systemdata} 
and \ref{assumption antisymmetric} hold and
the magnitudes of the
 first derivatives of $b^{i}$ are
bounded by $K$
 and
$u^{(1)}$ and $u^{(2)}$ are  
solutions,  then we can repeat  
the above argument with $p=2$ to get $u^{(1)}=u^{(2)}$.  
Thus we have proved uniqueness under weaker conditions than
the ones imposed in 
Theorems \ref{theorem systemmain} and \ref{theorem antisymmetric}.

To show the existence of  solutions we approximate the data of system \eqref{system} 
with smooth ones, satisfying also the strong stochastic parabolicity, Assumption 
\ref{assumption strongparabolicity}. To this end we will use the approximation described 
in the following lemma. 

\begin{lemma}                                                       \label{lemma smoothing}
Let Assumptions \ref{assumption systemcoeff} and
\ref{assumption systemparabolicity} 
(\ref{assumption antisymmetric}, respectively) hold with $m\geq1$. 
Then for every $\varepsilon\in(0,1)$ 
there exist $\cP\otimes\cB(\bR^d)$-measurable smooth (in $x$) functions 
$a^{\varepsilon ij}$, 
$b^{(\varepsilon)i}$, $c^{(\varepsilon)}$, 
$\sigma^{(\varepsilon)i }$, 
$\nu^{(\varepsilon) },D_ka^{ \varepsilon ij}$ and 
$h^{(\varepsilon)i}$,  
satisfying the 
following conditions for every $i,j,k=1,...,d$.
\begin{enumerate}
\item[(i)] There is a constant $N=N(K)$ such that  
$$
|a^{ \varepsilon ij}-a^{ij}|+|b^{(\varepsilon)i}-b^{i}|+|c^{(\varepsilon)}-c|
+|D_ka^{ \varepsilon ij}-D_ka^{ij}|\leq N\varepsilon,
$$
$$
|\sigma^{(\varepsilon)i }-
\sigma^{i }|+
|\nu^{(\varepsilon) }-\nu |\leq N\varepsilon
$$ 
for all $(\omega,t,x)$ and $i,j,k=1,...,d$. 
\item[(ii)] For every integer $n\geq0$ the partial derivatives in $x$ 
of $a^{ \varepsilon ij}$, $b^{(\varepsilon)i}$, $c^{(\varepsilon)}$, 
$\sigma^{(\varepsilon)i}$ and 
$\nu^{(\varepsilon)}$ 
 up to order $n$ are $\cP\otimes\cB(\bR^d)$-measurable functions, 
in magnitude bounded by a constant. For $n=m$ this constant is independent 
of $\varepsilon$, it depends only on $m$, $M$, $d$ and $K$; 
\item[(iii)] For the matrix 
$\alpha^{ \varepsilon ij}
:=2a^{ \varepsilon ij}-\sigma^{(\varepsilon)ik}\sigma^{(\varepsilon)jk}
$ 
we have 
$$
\alpha^{ \varepsilon ij}\lambda^{i}\lambda^{j}\geq \varepsilon\sum_{i=1}^d|\lambda^i|^2
\quad\text{for all $\lambda=(\lambda^1,...,\lambda^d)\in\bR^d$};
$$
\item[(iv)] Assumption \ref{assumption systemparabolicity} 
(\ref{assumption antisymmetric}, respectively)
holds for the functions 
$\alpha^{ \varepsilon ij}$, 
$\beta^{ \varepsilon i}
:=b^{(\varepsilon)i}-\sigma^{(\varepsilon)ik}\nu^
{(\varepsilon)k}$ 
and $h^{(\varepsilon)i}$  
in place of $\alpha^{ij}$,  
$\beta^i$ and $h^i$,  respectively,  with 
the same constant $K_0$.  
\end{enumerate}
\end{lemma}

\begin{proof}
The proofs of the two statements containing 
Assumptions \ref{assumption systemparabolicity} 
and \ref{assumption antisymmetric}, respectively, go in essentially the same way, 
therefore we only detail the former. Let $\zeta$ be a nonnegative smooth function on $\bR^d$ 
with unit integral and support in the unit ball, 
and let $\zeta_{\varepsilon}(x)=\varepsilon^{-d}\zeta(x/\varepsilon)$. 
Define 
$$
b^{(\varepsilon)i}=b^i\ast\zeta_{\varepsilon},\,
c^{(\varepsilon)}=c\ast\zeta_{\varepsilon},\,
\sigma^{(\varepsilon)i}=\sigma^{i}\ast\zeta_{\varepsilon},\,
\nu^{(\varepsilon)}=\nu\ast\zeta_{\varepsilon},\,
h^{(\varepsilon)i}=h^i\ast\zeta_{\varepsilon}, 
$$
and $ a ^{ \varepsilon ij}=a^{ij}\ast\zeta_{\varepsilon}+k\varepsilon\delta_{ij}$ 
with a constant $k>0$ determined later, where $\delta_{ij}$ is the Kronecker symbol and 
`$\ast$' means the convolution in the variable $x\in\bR^d$. Since we have mollified 
functions which are bounded and Lipschitz continuous, 
the mollified functions, together with $a^{ \varepsilon ij}$ and 
$D_ka^{ \varepsilon ij}$,  
satisfy conditions (i) and (ii). Furthermore, 
$$
|\sigma^{(\varepsilon)ir}\nu^{(\varepsilon)r}-\sigma^{ir}\nu^{r}|
\leq|\sigma^{(\varepsilon)i}-\sigma^i||\nu^{(\varepsilon)}|+
|\sigma^i||\nu^{(\varepsilon)}-\nu|\leq2K^2\varepsilon, 
$$
for every $i=1,...,d$. 
Similarly, 
$$
|\sigma^{(\varepsilon)ir}\sigma^{(\varepsilon)jr}
-\sigma^{ir}\sigma^{jr}|\leq2K^2\varepsilon,
\quad 
|b^{(\varepsilon)i}-b^{i}|\leq K\varepsilon,
\quad
|h^{(\varepsilon)i}-h^i|\leq N\varepsilon 
$$
for all $i,j=1,2,...,d$. Hence setting 
$$
B^{ \varepsilon i}=
b^{(\varepsilon)i}-\sigma^{(\varepsilon)ik}\nu^{(\varepsilon)k}-h^{(\varepsilon)i}I_M,
$$
and using the notation $B^i$ for the same expression without 
the superscript 
`$ \varepsilon $',  
we 
have 
$$
|B^{ \varepsilon i}-B^{i}|\leq |b^{(\varepsilon)i}-b^{i}|
+|\sigma^{(\varepsilon)ir}\nu^{(\varepsilon)r}-\sigma^{ir}\nu^{r}|
+\sqrt{M}|h^{(\varepsilon)i}-h^i|\leq R\varepsilon, 
$$
$$
|B^{(\varepsilon)i}+B^{i}|\leq R
$$
with a constant $R=R(M,K)$. Thus for any $z_1$,...,$z_d$ vectors 
from $\bR^M$ 
$$
|\langle B^{ \varepsilon i} z_i\rangle^2-\langle B^{i} z_i\rangle^2|
=|\langle (B^{ \varepsilon i} -B^{i})z_i,
(B^{ \varepsilon j}+B^{j})z_j\rangle|
$$
$$
\leq |B^{ \varepsilon i} -B^{i}||B^{ \varepsilon j}+B^{j}|
\langle z_i\rangle\langle z_j\rangle
\leq dR^2\varepsilon\sum_{i=1}^d\langle z_i\rangle^2.  
$$
Therefore  
$$
\langle B^{ \varepsilon i} z_i\rangle^2
\leq 
\langle B^i z_i\rangle^2+C_1\varepsilon \sum_{i=1}^d\langle z_i\rangle^2 
$$
with a constant $C_1=C_1(M,K,d)$. Similarly, 
$$
\sum_{i,j}(2a^{ \varepsilon ij}-\sigma^{(\varepsilon)ik}\sigma^{(\varepsilon)jk})\langle z_i,z_j\rangle
$$
$$
\geq\sum_{i,j}(2a^{ij}-\sigma^{ik}\sigma^{jk})\langle z_i,z_j\rangle
+(k-C_2)\varepsilon \sum_i\langle z_i\rangle^2 
$$
with a constant $C_2=C_2(K,m,d)$. Consequently, 
$$
\langle(\beta^{ \varepsilon i}-h^{(\varepsilon)i}I_M)z_i\rangle^2
\leq
\langle B^i z_i\rangle^2+C_1\varepsilon \sum_{i=1}^d\langle z_i\rangle^2
$$
$$
\leq K_0\sum_{i,j=1}^d\alpha^{ij}\langle z_i,z_j\rangle
+C_1\varepsilon \sum_{i=1}^d\langle z_i\rangle^2
$$
$$
\leq K_0\sum_{i,j=1}^d\alpha^{ \varepsilon ij}\langle z_i,z_j\rangle
+(K_0(C_2-k)+C_1)\varepsilon \sum_{i=1}^d\langle z_i\rangle^2. 
$$
Choosing $k$ such that $K_0(C_2-k)+C_1=-K_0$ we get 
$$
\langle(\beta^{ \varepsilon i}-h^{(\varepsilon)i}I_M)z_i\rangle^2
+K_0\varepsilon \sum_{i=1}^d\langle z_i\rangle^2
\leq K_0\sum_{i,j=1}^d\alpha^{ \varepsilon ij}\langle z_i,z_j\rangle.  
$$
Hence statements (iii) and (iv) follow immediately. 
\end{proof}

Now we start with the proof of the existence of  
solutions which are   $W^m_p(\bR^{d},\bR^{M})$-valued 
if the Assumptions \ref{assumption systemcoeff}, 
\ref{assumption systemdata} 
and \ref{assumption systemparabolicity} hold with 
$m\geq1$. 
First we make the additional assumptions 
that $\psi$, $f$ and $g$ vanish for $|x|\geq R$ for some $R>0$, 
and that  $q\in[2,\infty)$ and 
\begin{equation}                                     \label{4.14.3.14}
E|\psi|_{W^m_p}^{q}+E\mathcal K^{q}_{m,q}(T)<\infty.  
\end{equation}
 For each $\varepsilon>0$ we consider the system 
$$                                          
du^{\varepsilon}_{t}=
[\sigma^{(\varepsilon)ir}_{t}D_{i}u^{\varepsilon}_{t}
+\nu^{(\varepsilon)r}_{t}u^{\varepsilon}_{t}
+g^{(\varepsilon)r}_{t}]\,dw^{r}_{t}
$$
\begin{equation}                                         \label{2.13.3.14}
+\big[ a^{ \varepsilon ij}_{t}D_{ij}u^{\varepsilon}_{t}
+b^{(\varepsilon)i}_{t}D_{i} u^{\varepsilon}_{t} 
 +f^{(\varepsilon)}_{t}\big]\,dt 
\end{equation}  
with initial condition
\begin{equation}                                                                     \label{3.13.3.14}
u_0^{(\varepsilon)}=\psi^{(\varepsilon)}, 
\end{equation}
where the coefficients are taken from Lemma \ref{lemma smoothing}, 
and 
$\psi^{(\epsilon)}$, $f^{(\epsilon)}$ and $g^{(\epsilon)}$ are defined 
as the convolution of $\psi$, $f$ and $g$, 
respectively,  
with $\zeta_{\varepsilon}(\cdot)=\varepsilon^{-d}\zeta(\cdot/\varepsilon)$ 
for  $\zeta\in C^{\infty}_{0}(\bR^d)$   
taken from the proof of
 Lemma \ref{lemma smoothing}. 
By Theorem \ref{theorem L_2} the above equation has a unique solution 
$u^{\varepsilon}$, 
which is a $W^n_2(\bR^{d},\bR^{M})$-valued continuous process for all $n$. 
Hence, by Sobolev embeddings, $u^{\varepsilon}$ 
is a $W^{m+1}_p(\bR^{d},\bR^{M})$-valued 
continuous process, and therefore we can use 
Lemma \ref{lemma estimate} to get
\begin{equation}                                                                 \label{1.13.3.14}
E\sup_{t\in[0,T]}|u^{\varepsilon}_t|_{W^n_{p' }}^{q}
\leq N(E|\psi^{(\varepsilon)}|_{W^n_{p' }}^{q}
+E(\mathcal K^{ \varepsilon }_{n, p' })^{q}(T)) 
\end{equation}
for $p'\in\{p,2\}$ and $n=0,1,2...m$, where $K^{ \varepsilon }_{n,{p'}}$ 
is defined by \eqref{02.25} with $f^{(\varepsilon)}$ 
and $g^{(\varepsilon)}$ in place of $f$ and $g$, respectively. 
Keeping in mind that $T^{1/r}\leq\max\{1,T\}$, 
and using basic properties of convolution, 
we can conclude that 
\begin{equation}                                                                  \label{1.17.3.14}
E\left(\int_0^T|u^{\varepsilon}_t|_{W^n_{p'}}^r\,dt\right)^{{q}/r}
\leq N(E|\psi|_{W^n_{p'}}^{q}+E\mathcal K^{q}_{n,{p'}}(T))
\end{equation}
for any $r>1$ and with $N=N(m,p,q,d,M,K,T)$ not depending on $r$. 

For  integers $n\geq0$, and  any $r,q\in(1,\infty)$ 
 let $\bH^n_{p,r,q}$ be the space of 
$\bR^M$-valued functions $v=v_t(x)=(v^i_t(x))_{i=1}^M$ on 
$\Omega\times[0,T]\times\bR^d$ such that $v=(v_t(\cdot))_{t\in[0,T]}$ 
are 
$W^n_p(\bR^{d},\bR^{M})$-valued predictable 
processes and 
$$
|v|^q_{\bH^n_{p,r,q}}=E\left(\int_0^T|v_t|^r_{W^n_p}\,dt\right)^{q/r}<\infty. 
$$
Then $\bH^n_{p,r,q}$ with the norm defined above is a 
reflexive Banach space for each $n\geq0$ and 
$p,r,q\in(1,\infty)$. 
We use the notation $\bH^n_{p,q}$ for $\bH^n_{p,q,q}$.

By  Assumption \ref{assumption systemdata}  
the right-hand side of \eqref{1.17.3.14} is finite for 
$p'=p$ and  also for $p=2$ 
since $\psi$, $f$ and $g$ vanish for $|x|\geq R$.  
Thus there exists a sequence $(\varepsilon_k)_{k\in\mathbb{N}}$ 
such that $\epsilon_k\rightarrow 0$ and for $p'=p,2$ and 
integers $r>1$ and  
$n\in[0,m]$ 
the sequence $v^k:=u^{ \varepsilon_k }$ 
converges weakly in $\bH^n_{{p'}r,q}$ to some $v\in H^m_{p',r,q}$, 
which therefore also satisfies
$$
E\left(\int_0^T|v_t|_{W^n_{p'}}^r\,dt\right)^{q/r}
\leq N(E|\psi|_{W^n_{p'}}^{q}+E\mathcal K^{q}_{n,q}(T)) 
$$
for $p'=p,2$ and integers $r>1$. 
Using this with $p'=p$ and letting $r\rightarrow\infty$ 
by Fatou's lemma we obtain
\begin{equation}                                   \label{esssup}
E\esssup_{t\in[0,T]}|v_t|^q_{W^n_{p}}
\leq N(E|\psi|_{W^n_{p}}^{q}
+E\mathcal K^{q}_{n,p}(T))\quad\text{for $n=0,1,...,m$}.
\end{equation}

Now we are going to show that a suitable stochastic modification of $v$ is a solution 
of \eqref{system}-\eqref{systemini}. To this end we fix an $\bR^M$-valued function  
$\varphi$  in $C_0^{\infty}(\bR^d)$   
and a predictable  real-valued process 
$(\eta_t)_{t\in[0,T]}$, which is bounded by some 
constant $C$,  
and define the functionals $\Phi$, $\Phi_k$, $\Psi$ 
and $\Psi_k$ over $\bH^1_{p,q}$ by 
$$
\Phi_k(u)=E\int_0^T\eta_t
\int_0^t\{-(a^{ \varepsilon_k ij}_sD_iu_s,D_j\varphi)                                                
+(\bar b^{ \varepsilon_k i}_sD_iu_s+c^{(\varepsilon_k)}_su_s,\varphi)\}\,ds\,dt, 
$$
$$
\Phi(u)=E\int_0^T\eta_t
\int_0^t\{-(a^{ij}_sD_iu_s,D_j\varphi)+(\bar b^{i}_sD_iu_s+c_su_s,\varphi)\}\,ds\,dt,                                                
$$
$$
\Psi(u)=E\int_0^T\eta_t\int_0^t(\sigma^{ir}_{t}D_{i}u_{t}+\nu^{r}_{t}u_{t},
\varphi)\,dw^{r}_{t}\,dt
$$
$$
\Psi_k(u)=E\int_0^T\eta_t\int_0^t(\sigma^{(\varepsilon_k)ir}_{t}D_{i}u_{t}
+\nu^{(\varepsilon_k)r}_{t}u_{t},
\varphi)\,dw^{r}_{t}\,dt
$$
for $u\in\bH^1_{p,q}$ for each $k\geq1$, where  
$\bar b^{ \varepsilon    i}=b^{(\varepsilon)i}-D_ja^{ \varepsilon ij}I_{M}$.
 By the Bunyakovsky-Cauchy-Schwarz and 
the Burkholder-Davis-Gundy inequalities for all $u\in 
\bH^1_{p,q}$ we have
$$
\Phi(u)\leq CNT^{2-1/q}|u|_{\bH^1_{p,q}}|\varphi|_{W^1_{\bar p}}, 
$$
\begin{align*}
\Psi(u)\leq&CTE\sup_{t\leq T}|\int_0^t
(\sigma^{ir}_{t}D_{i}u_{t}+\nu^{r}_{t}u_{t},
\varphi)\,dw^{r}_{t}|
\\
\leq& 3CT
E\left\{\int_0^T\sum_{r=1}^{\infty}
(\sigma^{ir}_{t}D_{i}u_{t}+\nu^{r}_{t}u_{t},\varphi)^2\,dt
\right\}^{1/2}\\
\leq&
3CTE\left\{\int_0^T\left(\int_{\bR^d}
|\langle\sigma^{ir}_{t}D_{i}u_{t}+\nu^{r}_{t}u_{t},\varphi\rangle|_{l_2}\,dx\right)^2\,dt
\right\}^{1/2}\\
\leq&CTNE\left\{\int_0^T|u_t|^2_{W^1_{p}}|\varphi|^2_{W^1_{\bar p}}\,dt\right\}^{1/2}
\leq CNT^{q/2}|u|_{\bH^1_{p,q}}|\varphi|_{W^1_{\bar p}}\\
\end{align*}
with a constant $N=N(K,d,M)$, where $\bar p=p/(p-1)$.   
(In the last inequality we make use of the assumption $q\geq2$.)
Consequently, 
$\Phi$  and $\Psi$ are continuous 
linear functionals over $\bH^1_{p,q}$, 
and 
therefore 
\begin{equation}                                \label{5.29.10.12}
\lim_{k\to\infty}\Phi(v^{k})=\Phi(v),
\quad
\lim_{k\to\infty}\Psi(v^{k})=\Psi(v). 
\end{equation} 
Using statement (i) of Lemma \ref{lemma smoothing}, 
we get 
\begin{equation}                              \label{6.29.10.12}     
|\Phi_k(u)-\Phi(u)|+
|\Psi_k(u)-\Psi(u)|
\leq N\varepsilon_k|u|_{\bH^1_{p,q}}|\varphi|_{W^1_{\bar p}}
\end{equation}
for all $u\in \bH^1_{p,q}$ with a constant $N=N(k,d,M)$. 
Since $u^{\varepsilon}$ is the solution of   \eqref{2.13.3.14}-\eqref{3.13.3.14}, 
we have 
$$
E\int_0^T\eta_t(v_t^{k},\varphi)\,dt
=E\int_0^T\eta_t(\psi^{k},\varphi)\,dt
+\Phi(v^{k})+\Psi(v^{k})
$$
\begin{equation}                                                             \label{7.13.3.14}
+F(f^{(\varepsilon_k)})+G(g^{(\varepsilon_k)}) 
\end{equation}
for each $k$, where 
$$
F(f^{(\varepsilon_k)})=E\int_0^T\eta_t\int_0^t(f_s^{(\varepsilon_k)},\varphi)\,ds\,dt,
$$
$$
G(g^{(\varepsilon_k)})=E\int_0^T\eta_t\int_0^t(g^{(\varepsilon_k)r}_s,\varphi)\,dw_s^r\,dt. 
$$
Taking into account that $|v^{k}|_{\bH^1_{p,q}}$ is a bounded sequence, 
from \eqref{5.29.10.12} and \eqref{6.29.10.12}
we obtain 
\begin{equation}                            \label{2.30.10.12}
\lim_{k\to\infty}\Phi_{n}(v^{k})=\Phi(v), 
\quad 
\lim_{k\to\infty}\Psi_{k}(v^{k})=\Psi(v). 
\end{equation}
One can see similarly (in fact easier), that 
\begin{equation}                                                                  \label{4.13.3.14}
\lim_{k\to\infty}E\int_0^T\eta_t(v^{k}_t,\varphi)\,dt
=E\int_0^T\eta_t(v_t,\varphi)\,dt,
\end{equation}
\begin{equation}                                                                  \label{5.13.3.14}
\lim_{k\to\infty}
E\int_0^T\eta_t(\psi^{(\varepsilon_k)}_t,\varphi)\,dt
=E\int_0^T\eta_t(\psi,\varphi)\,dt,
\end{equation}
\begin{equation}                                                                  \label{6.13.3.14}
\lim_{k\to\infty}F(f^{(\varepsilon_k)})=F(f),\quad 
\lim_{k\to\infty}G(g^{(\varepsilon_k)})=G(g).   
\end{equation}

Letting $k\to\infty$ in \eqref{7.13.3.14},  
and using  \eqref{2.30.10.12} through \eqref{6.13.3.14} 
we obtain 
$$
E\int_0^T\eta_t(v_t,\varphi)\,dt 
$$
$$    
=E\int_0^T\eta_t\Big\{(\psi,\varphi)
+\int_0^t\big[-(a^{ij}_sD_iu_s,D_j\varphi)+
(\bar b^{i}_sD_iu_s+c_su_s+f_s,\varphi)\big]\,ds
$$
$$
+\int_0^t(\sigma^{ir}D_iv_s+\nu^rv_s,\varphi)\,dw_s^r\Big\}\,dt
$$
for every bounded predictable process $(\eta_t)_{t\in[0,T]}$ and 
$\varphi$ from $C_0^{\infty}$.  
Hence for each $\varphi\in C_0^{\infty}$
$$
(v_t,\varphi)  
=(\psi,\varphi)
+\int_0^t\big[-(a^{ij}_sD_iv_s,D_j\varphi)
+(\bar b^{i}_sD_iv_s+c_sv_s+f_s,\varphi)\big]\,ds
$$
$$
+\int_0^t(\sigma^{ir}D_iv_s+\nu^rv_s+g^r_s,\varphi)\,dw_s^r
$$
holds for $P\times dt$ almost every 
$(\omega,t)\in\Omega\times[0,T]$.  Substituting here 
$(-1)^{|\alpha|}D^{\alpha}\varphi$ 
in place of $\varphi$ for a multi-index 
$\alpha=(\alpha_1,...,\alpha_d)$ of length 
$|\alpha|\leq m-1$ and integrating by parts, we see that 
\begin{equation}                                                                              \label{1.14.3.14}
(D^{\alpha}v_t,\varphi)  
=(D^{\alpha}\psi,\varphi)
+\int_0^t\big[-(F^j_s,D_j\varphi)+(F^0_s,\varphi)\big]\,ds
+\int_0^t(G^r_s,\varphi)\,dw_s^r
\end{equation}
for $P\times dt$ almost every 
$(\omega,t)\in\Omega\times[0,T]$, where, 
owing to the fact that \eqref{esssup} also holds with $2$
in  place of $p$, $F^i$ and $(G^{r})_{r=1}^{\infty}$  
are predictable processes with values in $L_2$-spaces for $i=0,1,...,d$, such that 
$$
\int_0^T\big(\sum_{i=0}^d|F^i_s|^2_{L_2}+|G_s|^2_{L_2}
\big)\,ds<\infty \quad\text{(a.s.)}. 
$$
Hence the theorem 
on It\^o's formula from \cite{KrR1981}  implies that in the equivalence 
class of $v$ in $\bH^m_{2,q}$ there is a 
$W^{m-1}_2(\bR^{d},\bR^{M})$-valued continuous 
 process, $u=(u_t)_{t\in[0,T]}$, and  
\eqref{1.14.3.14} with $u$ in place of $v$ holds 
for any $\varphi\in C_0^{\infty}(\bR^d)$
almost surely for all $t\in[0,T]$. After that
an application of Lemma \ref{lemma itoformula} 
to $D^{\alpha}u$ for $|\alpha|\leq m-1$ yields that 
$D^{\alpha}u$ is an $L_p(\bR^{d},\bR^{M})$-valued, 
strongly continuous process 
for every $|\alpha|\leq m-1$, i.e.,  
$u$ is a $W^{m-1}_p(\bR^{d},\bR^{M})$-valued 
strongly continuous process. This, \eqref{esssup},
and the denseness of $C_{0}^{\infty}$ in 
$W^{m }_p(\bR^{d},\bR^{M})$ implies that (a.s.) $u$ is a
$W^{m}_p(\bR^{d},\bR^{M})$-valued 
weakly continuous process and \eqref{estimate of main theorem}
holds.

To prove the theorem without the assumption that
  $\psi$, $f$ and $g$ have compact support,  
 we take a 
$\zeta\in C^{\infty}_{0}(\bR^{d})$
such that
$\zeta(x)=1$ for $|x|\leq1$ and $\zeta(x)=0$ for $|x|\geq2$, 
and define $\zeta_{n}(\cdot)=\zeta(\cdot/n)$ for $n>0$. 
Let $u(n)=(u_t(n))_{t\in[0,T]}$ denote the solution 
of \eqref{system}-\eqref{systemini} 
with
$\zeta_{n}\psi$, $\zeta_{n}f$ and $\zeta_{n}g$ in place of $\psi$, $f$ and $g$,
respectively. By virtue of what we have proved above, $u(n)$ is a weakly continuous 
$W^{m}_{p}(\bR^{d},\bR^{M})$-valued process,  and 
$$
      E\sup_{t\in[0,T]}|u_t(n)-u_t(l)|^{q}_{W^m_p}
\leq N E|(\zeta_{n}-\zeta_{l})\psi |^{q}_{W^m_p}
$$
$$
+ NE \big(
\int_{0}^{T}\{|(\zeta_n-\zeta_l)f_s|^p_{W^m_p}
+|(\zeta_{n}-\zeta_{l})g_s|_{W^{m+1}_p }^{p }\}\,ds
\big)^{q/p}.
$$
Letting here $n,l\to\infty$ and applying Lebesgue's theorem on dominated convergence  
in the left-hand side, we see that the right-hand side of the inequality 
tends to zero. Thus for a subsequence $n_k\to\infty$ we have that  
$u_t(n_k)$ converges 
strongly in $W^m_p(\bR^{d},\bR^{M})$, uniformly in $t\in[0,T]$, 
to a process $u$. 
Hence $u$ is a weakly continuous $W^m_p(\bR^{d},\bR^{M})$-valued process.
 It is easy to show that it solves 
\eqref{system}-\eqref{systemini} and satisfies \eqref{estimate of main theorem}.

By using a standard stopping time argument we can dispense with  condition 
\eqref{4.14.3.14}. Finally we can prove estimate \eqref{estimate of main theorem} 
for $q\in(0, 2 )$ 
by applying Lemma 3.2 from \cite{GK2003} in the same way 
as it is used there to prove the corresponding 
estimate in the case $M=1$. 
The proof of the Theorem \ref{theorem systemmain} is complete. 
We have already showed the uniqueness statement of Theorem 
\ref{theorem antisymmetric}, the proof of the other assertions 
goes in the above way with obvious changes.  

\noindent 
 {\bf Acknowledgement.} The results of this paper were presented at the 
 9th International Meeting on ``Stochastic Partial Differential Equations and Applications"  
in Levico Terme in Italy, in January, 2014, and at the meeting on ``Stochastic Processes and
Differential Equations in Infinite Dimensional Spaces" 
in King's College London, in March, 2014. The authors would like to thank the organisers  
for these possibilities.

\end{document}